\documentclass[12pt]{article}
\usepackage[latin1]{inputenc}
\usepackage{amsmath,amsthm,amssymb}
\usepackage{amsfonts}
\usepackage{amsmath,amsthm,amssymb,amscd}
\usepackage{latexsym}
\usepackage{color}
\usepackage{color,enumitem,graphicx}
\usepackage[colorlinks=true,urlcolor=blue,
citecolor=red,linkcolor=blue,linktocpage,pdfpagelabels,
bookmarksnumbered,bookmarksopen]{hyperref}
\usepackage{graphicx}
\usepackage{mathrsfs}

\makeatletter \@addtoreset{equation}{section} \makeatother

\newtheorem{theorem}{Theorem}[section]

\newtheorem{lemma}{Lemma}[section]
\newtheorem{remark}{Remark}[section]

\numberwithin{equation}{section}

\allowdisplaybreaks

\begin{document}

\title{On a nonhomogeneous Kirchhoff type elliptic system  with the singular Trudinger-Moser growth}%[Kirchhoff systems with exponential growth nonlinearity]
\author%[S. Deng, X. Tian]
{Shengbing Deng\footnote{
E-mail address:\, {\tt shbdeng@swu.edu.cn} (S. Deng),   {\tt xltianswumaths@163.com} (X. Tian)}   \, and Xingliang Tian\\
\footnotesize  School of Mathematics and Statistics, Southwest University,
Chongqing, 400715, P.R. China}
\date{ }
\maketitle

\begin{abstract}
{The aim of this paper is to study the multiplicity of solutions for the following Kirchhoff type elliptic systems
 \begin{eqnarray*}
    \left\{ \arraycolsep=1.5pt
       \begin{array}{ll}
        -m\left(\sum^k_{j=1}\|u_j\|^2\right)\Delta u_i=\frac{f_i(x,u_1,\ldots,u_k)}{|x|^\beta}+\varepsilon h_i(x),\ \ & \mbox{in}\ \ \Omega, \ \ i=1,\ldots,k ,\\[2mm]
        u_1=u_2=\cdots=u_k=0,\ \ & \mbox{on}\ \ \partial\Omega,
        \end{array}
    \right.
    \end{eqnarray*}
where $\Omega$ is a bounded domain in $\mathbb{R}^2$ containing the origin with smooth boundary,  $\beta\in [0,2)$, $m$ is a Kirchhoff type function, $\|u_j\|^2=\int_\Omega|\nabla u_j|^2dx$, $f_i$ behaves like $e^{\beta s^2}$ when $|s|\rightarrow \infty$ for some $\beta>0$, and there is $C^1$ function $F: \Omega\times\mathbb{R}^k\to \mathbb{R}$ such that $\left(\frac{\partial F}{\partial u_1},\ldots,\frac{\partial F}{\partial u_k}\right)=\left(f_1,\ldots,f_k\right)$, $h_i\in \left(\big(H^1_0(\Omega)\big)^*,\|\cdot\|_*\right)$.
We establish sufficient conditions for the multiplicity of solutions of the above system by using
variational methods with  a suitable singular Trudinger-Moser inequality when $\varepsilon>0$ is small.}


\end{abstract}

\section{{\bfseries Introduction}}

In last decades, a great attention has been focused on the study of problems involving exponential growth nonlinearities, which is related to the famous Trudinger-Moser inequality. Let $\Omega$ be a bounded domain in $\mathbb{R}^2$, and denote with $H_0^{1}(\Omega)$ the standard first order Sobolev space given by
\[
H_0^{1}(\Omega)=cl\Big\{u\in C^\infty_0(\Omega)\ :\ \int_\Omega|\nabla u|^2{\rm d}x<\infty\Big\},\quad\ \ \|u\| =\left(\int_\Omega|\nabla u|^2{\rm d}x\right)^{\frac{1}{2}}.
\]
This space is a limiting case for the Sobolev embedding theorem, which yields $H_0^{1}(\Omega)\hookrightarrow L^p(\Omega)$ for all $1\leq p<\infty$, but one knows by easy examples that $H_0^{1}(\Omega)\not\subseteq L^\infty(\Omega)$ such as, $u(x)=\log(1-\log|x|)$ in $B_1(0)$. Hence, one is led to look for a function $g(s):\mathbb{R}\to\mathbb{R}^+$ with maximal growth such that
\[
\sup\limits_{u\in H_0^{1}(\Omega),\|u\| \leq 1}\int_\Omega g(u){\rm d}x<\infty.
\]
It was shown by Trudinger \cite{Trudinger} and Moser \cite{m} that the maximal growth is of exponential type. More precisely, named the Trudinger-Moser inequality that
\[
\exp(\alpha u^{2})\in L^1(\Omega),\quad \forall\ u\in H_0^{1}(\Omega),\ \ \forall\ \alpha>0,
\]
and
\begin{align*}%\label{tmi}
\sup\limits_{u\in H_0^{1}(\Omega),\|u\| \leq 1}\int_\Omega \exp(\alpha u^{2}){\rm d}x< \infty,\quad \mbox{if}\  \alpha\leq 4\pi,
\end{align*}
where $4\pi$ is the sharp constant in the sense that the supremum in the left is $\infty$ if  $\alpha >4\pi$.

    In order to treat the system problems, here we give some definitions. For all $1\leq p<\infty$, we define $L^p(\Omega,\mathbb{R}^k)$ as
    \[
    L^p(\Omega,\mathbb{R}^k):=\underbrace{L^p(\Omega)\times\cdots\times L^p(\Omega)}_{k},
    \]
    where $L^p(\Omega)$ is the standard $L^p$-space, and since
    \begin{equation}\label{eqk}
    \frac{1}{k}\left(\sum^k_{i=1}|u_i|^p\right)\leq\left(\sum^k_{i=1}|u_i|^2\right)^{\frac{p}{2}}\leq %2^{(p-1)(k-1)}\left(\sum^k_{i=1}|u_i|^p\right),
    k^p\left(\sum^k_{i=1}|u_i|^p\right),
    \end{equation}
    we can know that $L^p(\Omega,\mathbb{R}^k)$ is well defined and for $U\in L^p(\Omega,\mathbb{R}^k)$, we define $\|U\|_p=\left(\int_\Omega |U|^p dx\right)^{1/p}$ where $|U|=(\sum^k_{i=1}|u_i|^2)^{1/2}$. Moreover we denote
    \[
    H^1_0(\Omega,\mathbb{R}^k):=\underbrace{H_0^{1}(\Omega)\times\cdots\times H_0^{1}(\Omega)}_{k},
    \]
    be the Sobolev space modeled in $L^2(\Omega,\mathbb{R}^k)$ with the scalar product
    \begin{equation*}\
    \langle U,V\rangle=\sum^k_{i=1}\int_{\Omega}\nabla u_i\nabla v_idx,
    \end{equation*}
    where $U, V\in L^2(\Omega,\mathbb{R}^k)$, to which corresponds the norm $\|U\|=\langle U,U\rangle^{1/2}=(\sum^k_{i=1}\|u_i\|^2)^{1/2}$, then $H^1_0(\Omega,\mathbb{R}^k)$ is well defined and also is a Hilbert space. For all $1\leq p<\infty$,
    by the inequality (\ref{eqk}) and the Sobolev embedding theorem, we can know that the embedding $H^1_0(\Omega,\mathbb{R}^k)\hookrightarrow L^p(\Omega,\mathbb{R}^k)$ is compact and $H^1_0(\Omega,\mathbb{R}^k)\nsubseteq L^\infty(\Omega,\mathbb{R}^k)$, where $L^\infty(\Omega,\mathbb{R}^k):=L^\infty(\Omega)\times\cdots\times L^\infty(\Omega)$. In Section \ref{sec preliminaries}, we will establish the Trudinger-Moser type inequality for $H^1_0(\Omega,\mathbb{R}^k)$.

    Now, we begin to state our problem. Let $\Omega$ be a bounded domain in $\mathbb{R}^2$ containing the origin with smooth boundary,  we study the multiplicity of solutions for the following Kirchhoff type systems
    \begin{eqnarray}\label{P}
    \left\{ \arraycolsep=1.5pt
       \begin{array}{ll}
        -m\left(\sum^k_{j=1}\|u_j\|^2\right)\Delta u_i=\frac{f_i(x,u_1,\ldots,u_k)}{|x|^\beta}+\varepsilon h_i(x),\ \ & \mbox{in}\ \ \Omega, \ \ i=1,\ldots,k ,\\[2mm]
        u_1=u_2=\cdots=u_k=0,\ \ & \mbox{on}\ \ \partial\Omega,
        \end{array}
    \right.
    \end{eqnarray}
where $\beta\in [0,2)$,  $m$ is a continuous Kirchhoff type function, $h_i\in \big(\big(H^1_0(\Omega)\big)^*,\|\cdot\|_*\big)\backslash\{0\}$ for some $i\in\{1,\ldots,k\}$, $\varepsilon$ is a small positive parameter, and $f_i$ has the maximal growth which allows treating (\ref{P}) variationally in the Sobolev space $H^1_0(\Omega,\mathbb{R}^k)$. We shall consider the variational situation in which
    \begin{equation*}
    (f_1(x,U),\ldots,f_k(x,U))=\nabla F(x,U)
    \end{equation*}
    for some function $F:\Omega \times \mathbb{R}^k \rightarrow \mathbb{R}$ of class $C^1$, where $\nabla F$ stands for the gradient of $F$ in the variables $U=(u_1,\ldots,u_k)\in \mathbb{R}^k$.
    We then rewrite ($\ref{P}$) in the matrix form as
    \begin{equation}\label{Pb1}
    -m\left(\|U\|^2\right)\Delta U=\frac{\nabla F(x,U)}{|x|^\beta}+\varepsilon H(x),
    \end{equation}
    where $\Delta U=(\Delta u_1,\ldots,\Delta u_k)$, $\frac{\nabla F(x,U)}{|x|^\beta}=\Big(\frac{f_1(x,U)}{|x|^\beta},\ldots,\frac{f_k(x,U)}{|x|^\beta}\Big)$ and $H(x)=\big(h_1(x),\ldots,h_k(x)\big)$.

System (\ref{P}) is called nonlocal because of the term $m\left(\sum^k_{j=1}\|u_j\|^2\right)$ which implies that the equation in (\ref{P}) is no longer a pointwise identity. As we will see later the presence of the term $m\left(\sum^k_{j=1}\|u_j\|^2\right)$ provokes some mathematical difficulties which makes the study of such a class of problems particularly interesting. Moreover, System (\ref{P}) with $k=1$ has a physical appeal which is generalization of a model introduced in 1883 by Kirchhoff \cite{k}.
There are many results about the existence and multiplicity of solutions for Kirchhoff problems by many mathematicians, we refer to \cite{acm,chen,cy2,c2,fs,fiscellaValdinoci,hezou,hezou2,nt} and the references therein.
When $k=1$, $\beta=0$ and $\varepsilon=0$, system (\ref{P}) become the following Kirchhoff type problem
\begin{eqnarray}\label{P1}
    \left\{ \arraycolsep=1.5pt
       \begin{array}{ll}
        -m\Big(\int_\Omega|\nabla u|^2dx\Big)\Delta u =f(x,u),\ \ & \mbox{in}\ \ \Omega,\\[2mm]
        u=0,\ \ & \mbox{on}\ \ \partial\Omega,
        \end{array}
    \right.
\end{eqnarray}
where the Kirchhoff function $m : \mathbb{R}_{+} \rightarrow \mathbb{R}_{+}$ satisfies

$(\overline{M}_1)$  there exists $m_0>0$ such that $m(t)\geq m_0$ for all $t\geq0$ and $M(t+s)\geq M(t)+M(s)$ for all $s,t\geq0$, where $M(t)=\int^t_0 m(\tau)d\tau$ is  the primitive of $m$.

$(\overline{M}_2)$  there exist constants $a_1,a_2>0$ and $t_0>0$ such that for some $\sigma\in\mathbb{R}$, $m(t) \leq a_{1}+a_{2} t^{\sigma}, \forall t \geq t_{0}$.

$(\overline{M}_3)$  $\frac{m(t)}{t}$ is nonincreasing for $t>0$.

\noindent Moreover, the nonlinearity $f:\Omega\times\mathbb{R}\to\mathbb{R}$ is continuous and satisfies

$(\overline{F}_1)$  there exist constants $s_0,K_0>0$ such that
$F(x, s) \leq K_{0} f(x, s), \ \forall(x, s) \in \Omega \times\left[s_{0},+\infty\right)$.

$(\overline{F}_2)$  for each $x \in \Omega, \frac{f(x, s)}{s^{3}}$ is increasing for $s>0$.

$(\overline{F}_3)$  there exists $\beta_{0}>\frac{2}{\alpha_{0} d^{2}} m\left(4 \pi / \alpha_{0}\right)$ such that
$\lim _{s \rightarrow+\infty} \frac{s f(x, s)}{\exp \left(\alpha_{0} s^{2}\right)} \geq \beta_{0}$  uniformly in  $x \in \Omega$.\\
Under these assumptions, by using minimax techniques with the Trudinger-Moser inequality, Figueiredo and Severo  \cite{fs} obtained
the existence of ground state solution of (\ref{P1}).
We note that  hypothesis $(\overline{F}_2)$ is necessary to obtain precise information about
the minimax level of the energy functional associated to problem (\ref{P1}),  they show
the existence of the least energy solution.
Recently, Naimen and  Tarsi \cite{nt} studied the existence and multiplicity of solutions for problem (\ref{P1}) with $m(t)=1+\alpha t$ under some weaker  assumptions than those in \cite{fs}.

On the other hand, we mention that the existence of solutions for elliptic equations involving critical exponential nonlinearities and a small nonhomogeneous term was considered by many authors, see \cite{adiyang,am,doms,lamlu4,y2012} and the references therein. In the whole Euclidean space $\mathbb{R}^N$, for $N$-Laplacian problems in \cite{doms},
for $N$-Laplacian problem with the nonlinear term involving critical Hardy exponential growth and the nonhomogeneous term in \cite{adiyang,y2012}. What's more, Lam and Lu \cite{lamlu4} established the existence and multiplicity of nontrivial solutions for the nonuniformly elliptic equations of $N$-Laplacian type. % and Hamiltonian-type systems in \cite{lamlu3}.
Moreover, Manasses de Souza \cite{ds} has studied the existence of solutions for a singular class of elliptic systems involving critical exponential growth in a bounded domain of $\mathbb{R}^2$. To the best of our knowledge, there are no results for (\ref{P}) with Kirchhoff function and exponential growth nonlinearity.

The  main purpose of the present paper is to consider the multiplicity of solutions of system (\ref{P}) and overcome the lack of compactness due to the presence of exponential growth terms as well as the degenerate nature of the Kirchhoff coefficient.

%    We say that $f:\Omega \times\mathbb{R}^2$ has subcritical growth at $+\infty$ if
%    \begin{equation}\label{1.2}
%    \lim_{|U|\to+\infty}\frac{|f(x,U)|}{e^{\alpha |U|^2}}=0,
%     \ \ \text{uniformly on }x\in\Omega,\ \ \forall\alpha >0\\[3pt]
%    \end{equation}
%and $f$ has critical growth at $+\infty$ if there exists $\alpha_0 >0$ such that\
%    \begin{equation}\label{1.3}
%    \lim_{|U|\to+\infty}\frac{|f(x,U)|}{e^{\alpha |U|^2}}
%    =  \begin{cases}
%     0,
%     & \text{uniformly on }x\in\Omega,\forall\alpha >\alpha_0,\\[3pt]
%%     +\infty,
%     & \text{uniformly on }x\in\Omega,\forall\alpha <\alpha_0.\\[3pt]
%    \end{cases}
%    \end{equation}

    Let us introduce the precise assumptions under which our problem is studied. For this, we define $M(t)=\int^  t_0 m(\tau)d\tau$, the primitive of $m$ so that $M(0)=0$. The hypotheses on Kirchhoff function $m:\mathbb{R}^+ \rightarrow\mathbb{R}^+$ are the following:

 ($M_1$)
    there exists $m_0>0$ such that $m(t)\geq m_0$ for all $t\geq 0$;

($M_2$)
    $m(t)$ is nondecreasing for $t\geq 0$;

($M_3$)
    there exists \ $\theta>1$ such that $\theta M(t)-m(t)t$ is nondecreasing for $t\geq 0$.

\begin{remark}\label{rem1}\rm
    By $(M_1)$, we can get that $M(t)$ is increasing for $t\geq 0$.
\end{remark}

\begin{remark}\label{rem3}\rm
From $(M_3)$,  we have that
    \begin{equation}\label{1.4}
    \theta \mathcal{M}(t)-M(t)t\geq 0,\ \ \forall t\geq 0.
    \end{equation}
\end{remark}

\begin{remark}\label{rem2}\rm
A typical example of a function $m$ satisfying the conditions $(M_1)-(M_3)$ is given by $m(t)=m_0+at^{\theta-1}$ with $\theta>1, m_0>0$ and $a\geq 0$. Another example is $m(t)=1+\ln(1+t)$.
\end{remark}

\begin{remark}\label{remcmm}\rm
Here, we compare assumptions $(\overline{M}_1)-(\overline{M}_3)$ in \cite{fs} as shown before with our present assumptions $(M_1)-(M_3)$.
From $(\overline{M}_3)$, we can obtain $(M_3)$ with $\theta=2$. Indeed, for any $0<t_1< t_2$,
    \begin{equation*}
    \begin{split}
    2M(t_1)-m(t_1)t_1&=2M(t_2)-2\int^{t_2}_{t_1}m(s)ds-\frac{m(t_1)t^2_1}{t_1} \\
    &\leq 2M(t_2)-\frac{m(t_2)(t^2_2-t^2_1)}{t_2}-\frac{m(t_2)t^2_1}{t_2} \\
    &=2 M(t_2)-m(t_2)t_2,
    \end{split}
    \end{equation*}
    thus $2M(t)-m(t)t$ is nondecreasing for $t\geq 0$.
    From $(M_1)-(M_2)$, we can obtain $(\overline{M}_1)$. Indeed, by $m(t)$ is nondecreasing for $t\geq 0$, we have $\int^{t+s}_t m(\tau)d\tau\geq \int^{s}_0 m(\tau)d\tau$ for all $s,t\geq0$, then it holds that $\int^{t}_0 m(\tau)d\tau+\int^{t+s}_t m(\tau)d\tau\geq \int^{t}_0 m(\tau)d\tau+\int^{s}_0 m(\tau)d\tau$, i.e. $M(t+s)\geq M(t)+M(s)$.
    Then from (\ref{1.4}), we can get $M(t)\geq M(1)t^\theta$ for $t\leq 1$, and $M(t)\leq M(1)t^\theta$ for $t\geq 1$, thus $M(t)\leq C_1t^\theta+C_2$ for some $C_1,C_2>0$.
\end{remark}

    Motivated by pioneer works of Adimurthi \cite{ad}, de Figueiredo et al. \cite{dmr1} and J.M. do \'{O} \cite{do}, we treat the so-called subcritical case and also the critical case. They say that a function $f:(\Omega,\mathbb{R})\to\mathbb{R}$ has subcritical growth on $\Omega\subset \mathbb{R}^2$ if
\[
\lim _{|u| \rightarrow \infty} \frac{|f(x, u)|}{\exp \left(\alpha u^{2}\right)}=0, \text { uniformly on } \Omega,\  \forall \alpha>0,
\]
and $f$ has critical growth on $\Omega$ if there exists $\alpha_0>0$ such that
\[
\lim _{|u| \rightarrow \infty} \frac{|f(x, u)|}{\exp \left(\alpha u^{2}\right)}=0, \text { uniformly on } \Omega,\  \forall \alpha>\alpha_0,
\]
and
\[
\lim _{|u| \rightarrow \infty} \frac{|f(x, u)|}{\exp \left(\alpha u^{2}\right)}=\infty, \text { uniformly on } \Omega,\  \forall \alpha<\alpha_0.
\]

    Throughout this paper, we assume the following hypotheses on the function $f_i:\Omega\times\mathbb{R}^k\rightarrow\mathbb{R}$ and $F$:

 ($F_0$) $f_i$ is continuous and $f_i(x,0,\ldots,0)=0$, $F(x,0,\ldots,0)=0$ uniformly on $x\in \Omega$.

  ($F_1$)
    $\lim \sup_{|U|\rightarrow 0} \frac {2F(x,U)}{|U|^2}<\lambda_1 m_0$ uniformly on $\Omega$, where
    $
    \lambda_1=\inf_{U\in H^1_0(\Omega,\mathbb{R}^k)\setminus\{0\}} \frac {\|U\|^2}{\int_{\Omega}|U|^2/|x|^\beta dx}>0;
    $

 ($F_2$)
    there exist constants $S_0,M_0>0$ such that
     $0<F(x,U)\leq M_0|\nabla F(x,U)|$, \ for all \ $|U|\geq S_0$ uniformly on $\Omega$;

($F_3$)
    there exists \ $\mu>2\theta$ such that
 $0<\mu F(x,U)\leq U\cdot\nabla F(x,U)$, \ for all \ $(x,U)\in \Omega\times\mathbb{R}^k\setminus\{\mathbf{0}\}$;

    We say that $U\in H^1_0(\Omega,\mathbb{R}^k)$ is a weak solution of problem (\ref{P}) it holds
    \begin{equation*}\
    m(\|U\|^2)\int_{\Omega}\nabla U\cdot\nabla \Phi dx=\int_{\Omega}\frac{\Phi\cdot \nabla F(x,U)}{|x|^\beta} dx+\varepsilon\int_{\Omega}\Phi\cdot H dx, \ \ \forall \ \Phi\in H^1_0(\Omega,\mathbb{R}^k)\\[3pt].
    \end{equation*}
    Since $f_i(x,0,\ldots,0)=0$,\ $U\equiv \mathbf{0}$ is the trivial solution of problem (\ref{P}). Thus, our aim is to obtain nontrivial solutions. Now, the main results of this work can state as follows.

    \begin{theorem}\label{thm1.2}
    Assume $f_i$ has subcritical growth at $\infty$,  that is,
    \begin{equation}\label{1.2}
    \lim_{|U|\to \infty}\frac{|f_i(x,U)|}{e^{\alpha |U|^2}}=0,
     \ \ \text{uniformly on }x\in\Omega,\ \ \forall\alpha >0.
    \end{equation}
    Moreover, assume $(M_1)$,\ $(M_3)$ and $(F_1)-(F_3)$, then there exists $\varepsilon_{sc}>0$ such that for each $0<\varepsilon<\varepsilon_{sc}$, problem (\ref{P}) has at least two nontrivial weak solutions. One of them with positive energy, while the other one with negative energy.
    \end{theorem}

    \begin{theorem}\label{thm1.3}
    Assume $f_i$ has critical growth at $\infty$, that is, if there exists $\alpha_0 >0$ such that
    \begin{equation}\label{1.3}
    \lim_{|u_i|\to\infty}\frac{|f_i(x,U)|}{e^{\alpha |U|^2}}
    =  \begin{cases}
     0,\ \ &\forall\alpha >\alpha_0,\\[3pt]
     +\infty,\ \ &\forall\alpha <\alpha_0,
    \end{cases}
    \end{equation}
    uniformly on $x\in\Omega$ and $u_j$ where $j\in\{1,\ldots,k\}\backslash\{i\}$. Moreover, suppose $(M_1)-(M_3)$, $(F_1)-(F_3)$ hold and

    $(F_4)$ if for some  $i\in \{1,\ldots,k\}$, there exists $\eta_0$ such that
    $$
    \liminf_{|u_i|\rightarrow\infty}\frac {u_i f_i(x,0,\ldots,0,u_i,0,\ldots,0)}{e^{\alpha_0 |u_i|^2}}\geq \eta_0>\frac { (2-\beta)^2m\left(\frac {2\pi(2-\beta)}{\alpha_0}\right)}{\alpha_0 d^{2-\beta} e },
    $$
    %with $U=(u_1,\ldots,u_k)\in \mathbb{R}^k$,
    uniformly on $\Omega$, %$B_d(x_0)$,
    where $d$ is the radius of the largest open ball contained in $\Omega$ centered at the origin. %the origin.
    Then there exists $\varepsilon_c>0$ such that for each $0<\varepsilon<\varepsilon_c$, problem (\ref{P}) has at least two nontrivial weak solutions. One of them with positive energy, while the other one with negative energy.
    \end{theorem}

    \begin{remark}\rm
    %For our problem, $(M_3)$ is well defined by W. Chen \cite{chen} which allows us to study more functions. Note that,
    If $\beta=\varepsilon=0,\ k=1$, for $(F_4)$, in \cite{fs}, the author replaced $e$ with 2, therefore, in order to get this improvement on the growth of the nonlinearity $f_i$ at $\infty$, it is crucial in our argument to use a new sequence in \cite{ddr}.
    \end{remark}

    %\begin{remark}\rm
    %The main difficulty encountered in nonlocal Kirchhoff problems is the competition that there is between the growths of $m$ and $f$. When $k=1,\ \beta=\varepsilon=0$, to overcome this trouble, the authors assume that $m$ is increasing or bounded in \cite{acm}, is different from the previous, Figueiredo and Severo \cite{fs} assume that $M(t+s)\geq M(t)+M(s)$ for all $s,t\geq 0$ and for each $x\in \Omega$, $\frac{f(x,s)}{s^3}$ is increasing for $s>0$ which they study there exists a positive ground state solution. In this present paper, we enhance a little condition of $m$ that $m$ is nondecreasing and weaken a lot conditions of $f$ which make us to study the existence of solutions of problem (\ref{P}), and the method mainly draws on Naimen and   Tarsi in \cite{nt}.
    %\end{remark}

 \begin{remark}\rm
 When $m\equiv 1$, $k=1$, $\beta=\varepsilon=0$, problems with critical growth involving the Laplace operator in bounded domains of $\mathbb{R}^2$ have been investigated in \cite{asy,ay,am,dmr1}, quasilinear elliptic problems with critical growth for $N$-Laplacian in bounded domains of $\mathbb{R}^N$ have been studied in \cite{ad,do}. Moreover, for the problems with critical growth in bounded domains in $\mathbb{R}^2$ and $f$ satisfied (see examples in \cite{ad,dmr1,do}) the asymptotic hypothesis
    \begin{align}\label{f41}
    \liminf_{|u|\rightarrow\infty}\frac {u f(x,u)}{e^{\alpha_0 u^2}}\geq \eta_0'>\frac {2}{\alpha_0 d^2},
    \end{align}
    and for Kirchhoff problem, in \cite{fs}
    \begin{align}\label{f42}
    \liminf_{|u|\rightarrow\infty}\frac {u f(x,u)}{e^{\alpha_0 u^2}}\geq \eta_0''>\frac {2 m\big(\frac {4\pi}{\alpha_0}\big)}{\alpha_0 d^2}.
    \end{align}
    What's more, when $m\equiv 1$, de Souza studied this problem in \cite{ds} and he assumed the hypothesis
    \begin{align}\label{f43}
    \lim\inf_{|U|\rightarrow\infty}\frac {u_i f_i(x,U)}{e^{2^{k-1}\alpha_0 |U|^2}}\geq \eta_0'''>\frac {(2-\beta)^2}{2^{k-1}\alpha_0 d^{2-\beta} e },
    \end{align}
    for some $i\in\{1,2,\ldots,k\}$.
    Motivated by \cite{as} and \cite{ra}, where they proved a version of Trudinger-Moser inequality with singular weight and studied the existence of positive weak solutions for the following semilinear and homogeneous elliptic problem
    \begin{eqnarray*}
    \left\{
       \begin{array}{ll}
        -\Delta u=\frac{f(x,u)}{|x|^\beta},\ \ & \mbox{in}\ \ \Omega, %\ \ i=1,\ldots,k ,
        \\[2mm]
        u=0,\ \ & \mbox{on}\ \ \partial\Omega.
        \end{array}
    \right.
    \end{eqnarray*}
    In the present paper, we improve and complement some of the results cited above for singular and nonhomogeneous case and extend the results to systems. Moreover, thanks to de Figueiredo, do \'{O} and Ruf of \cite{ddr} have constructed a proper sequence which makes the hypotheses (\ref{f41}) and (\ref{f42}) can be improved to $(F_4)$ in Theorem \ref{thm1.3}. And using the improvement of the Young's inequality which will be introduced in Lemma \ref{yi}, (\ref{f43}) can be improved better in $(F_4)$.
    \end{remark}

    \begin{remark}\rm
    On the basis of assumption $(F_0)$, if we further assume that for any $u_j\leq 0$ where $j\in \{1,\ldots,k\}$, $f_i(x,u_1,\ldots,u_k)\equiv 0$ for all $i\in \{1,\ldots,k\}$, uniformly in $x\in\Omega$, and $h_i\geq 0$ for all $i\in \{1,\ldots,k\}$, and $(F_4)$ changes to
    $$
    \liminf_{u_1,\ldots,u_k\rightarrow +\infty}\frac {U\cdot\nabla F(x,U)}{e^{\alpha_0 |U|^2}}\geq \eta_0>\frac { (2-\beta)^2m\left(\frac {2\pi(2-\beta)}{\alpha_0}\right)}{\alpha_0 d^{2-\beta} e },
    $$
    where $U=(u_1,\ldots,u_k)$, then by using Maximum principle, we can proof the solutions obtained in Theorem \ref{thm1.3} are entire positive, i.e. each of the component is positive. A typical example is $F(x,U)=|U|^\mu\exp(\alpha_0|U|^2)\prod^{k}_{i=1}{\rm sign} u_i$, where ${\rm sign} t=0$ if $t\leq 0$ and ${\rm sign} t=1$ if $t>0$.
    \end{remark}

    This paper is organized as follows:  Section \ref{sec preliminaries} contains some technical results. In Section \ref{vf}, we present the variational setting in which our problem will be treated. Section \ref{ps} is devoted to show some properties of the Palais-Samle sequences. Finally, we split Section \ref{main} into two subsections for the subcritical and critical cases, and we complete the proofs of our main results. Hereafter, $C,C_0,C_1,C_2...$ will denote positive (possibly different) constants.

\section{{\bfseries Some preliminary results}}\label{sec preliminaries}

    Now, we introduce some famous inequalities as follows, and inspired by those inequalities, we conclude some similar forms of inequalities. In this paper, we shall use the following version of the Trudinger-Moser inequality with a singular weight due to Adimurthi-Sandeep \cite{as}:

    \begin{lemma}\label{lemtm1}
    Let $\Omega$ be a bounded domain in $\mathbb{R}^2$ containing the origin and $u\in H^1_0(\Omega)$. Then for every $\alpha >0$, and $\beta \in [0,2)$,
    \begin{align*}
    \int_{\Omega} \frac{e^{\alpha |u|^2}}{|x|^\beta}<\infty.
    \end{align*}
    Moreover, there exists constant $C(\Omega)$ depending only on $\Omega$ such that
    \begin{align*}
    \sup_{||\nabla u||_2\leq 1}\int_{\Omega} \frac{e^{\alpha |u|^2}}{|x|^\beta}\leq C(\Omega),
    \end{align*}
    if and only if $\frac{\alpha}{4\pi}+\frac{\beta}{2} \leq 1$.
    \end{lemma}

    Then, we give two useful algebraic inequalities that will be used systematically in the rest of the paper as the following:
    \begin{lemma}\label{yi}
    ({\bfseries Improvement of the Young's inequality})
    Let $a_1,\ldots,a_k>0$, $p_1,\ldots,p_k>1$ with $\frac{1}{p_1}+\frac{1}{p_2}+\cdots+\frac{1}{p_k}=1$, then
    \begin{align}\label{yic}
    a_1 a_2 \cdots a_k\leq \frac{a^{p_1}_1}{p_1}+\frac{a^{p_2}_2}{p_2}+\cdots+\frac{a^{p_k}_k}{p_k}.
    \end{align}
    \end{lemma}

    \begin{proof}
    We will use the mathematical induction to proof this. When $k=2$, by Young  inequality, we can know this conclusion is correct. Suppose that when $k=s-1$ the conclusion is correct, we are going to show that when  $k=s$ the conclusion is still correct. Let
    \begin{align*}
    \frac{1}{q}=\sum^{s-1}_{i=1}\frac{1}{p_i},\ \ \frac{1}{q}+\frac{1}{p_s}=1,\ \ \mbox{then}\ \ \sum^{s-1}_{i=1}\frac{1}{p_i/q}=1.
    \end{align*}
    Thus,
    \begin{align*}
    \prod^s_{i=1} a_i=\Big(\prod^{s-1}_{i=1}a_i\Big)a_s\leq \frac{1}{q}\Big(\prod^{s-1}_{i=1}a_i\Big)^q+\frac{1}{p_s}a^{p_s}_s,
    \end{align*}
    by the mathematical induction, we can get that
    \begin{align*}
    \frac{1}{q}\Big(\prod^{s-1}_{i=1}a_i\Big)^q=\frac{1}{q}\Big(\prod^{s-1}_{i=1}a^{q}_i\Big)\leq \frac{1}{q}\sum^{s-1}\Big[\frac{1}{p_i/q}\big(a^q_i)^{p_i/q}\Big]=\sum^{s-1}\frac{1}{p_i}a^{p_i}_i.
    \end{align*}
    Therefore
    \begin{align*}
    \prod^s_{i=1} a_i\leq \sum^{s-1}_{i=1}\frac{1}{p_i}a^{p_i}_i+\frac{1}{p_s}a^{p_s}_s=\sum^{s}_{i=1}\frac{1}{p_i}a^{p_i}_i.
    \end{align*}
    This lemma is proved. If we take $p_1=p_2=\cdots=p_k=k$, we can get that
    \begin{align}\label{yiy}
    a_1 a_2 \cdots a_k\leq \frac{1}{k}\sum^{k}_{i=1}a^{k}_i\leq \sum^{k}_{i=1}a^{k}_i.
    \end{align}
    \end{proof}

    \begin{lemma}\label{yib}
    %({\bfseries Promotion of the Young's inequality})
    Suppose $a_1, a_2, \ldots,a_k\geq0$ with $a_1+a_2+\cdots+a_k<1$, then there exist $p_1,\ldots,p_k>1$ satisfying $\frac{1}{p_1}+\frac{1}{p_2}+\cdots+\frac{1}{p_k}=1$, such that
    \begin{align}\label{yibc}
    p_ia_i<1, \quad \mbox{for all}\ \ i=1,2,\ldots,k.
    \end{align}
    Moreover, if $a_1, a_2, \ldots,a_k\geq0$ satisfying $a_1+a_2+\cdots+a_k=1$, then we can take $p_i=\frac{1}{a_i}$ such that $\frac{1}{p_1}+\frac{1}{p_2}+\cdots+\frac{1}{p_k}=1$ and
    \begin{align}\label{yibcd}
    p_ia_i=1, \quad \mbox{for all}\ \ i=1,2,\ldots,k.
    \end{align}
    \end{lemma}

    \begin{proof}
    For the case $a_1, a_2, \ldots,a_k\geq0$ with $a_1+a_2+\cdots+a_k<1$. We also make use of the mathematical induction as previous. When $k=2$, $a_1, a_2\geq0$ with $a_1+a_2<1$. If $a_2=0$ (or $a_1=0$), then taking $p_1=\left(\frac{1}{2}+\frac{1}{2a_1}\right)>1$ (or $p_2=\left(\frac{1}{2}+\frac{1}{2a_2}\right)>1$), we can obtain $p_1a_1<1,\ p_2a_2<1$, where $p_2=\frac{p_1}{p_1-1}$. If $a_1, a_2>0$, then taking $p_1=\left(\frac{1}{2(1-a_1)}+\frac{1}{2a_2}\right)>1$, we can obtain $p_1a_1<1,\ p_2a_2<1$, where $p_2=\frac{p_1}{p_1-1}$. Suppose that when $k=s-1$ the conclusion is correct, if $k=s$ the conclusion is still correct, then the lemma follows. Let
    \begin{align*}
    (a_1+a_2+\cdots+a_{s-1})+a_s<1,
    \end{align*}
    then there exist $q_1, q_2>1$ satisfying $\frac{1}{q_1}+\frac{1}{q_2}=1$ such that
    \begin{align*}
    q_1(a_1+a_2+\cdots+a_{s-1})<1, \ \ q_2a_s<1.
    \end{align*}
    Then by the assumption, there exist $q_3, q_4,\ldots, q_{s+1}>1$ satisfying $\frac{1}{q_3}+\frac{1}{q_4}+\cdots+\frac{1}{q_{s+1}}=1$ such that
    \begin{align*}
    q_3(q_1a_1)<1,\ q_4(q_1a_2)<1, \ldots,\ q_{s+1}(q_1a_{s-1})<1,
    \end{align*}
    i.e.
    \begin{align*}
    (q_3q_1)a_1<1,\ (q_4q_1)a_2<1, \ldots,\ (q_{s+1}q_1)a_{s-1}<1.
    \end{align*}
    Taking $p_1=q_1q_3,\ p_2=q_1q_4,\ldots,\ p_{s-1}=q_1q_{s+1},\ p_s=q_2$, then it holds that
    \begin{align*}
    \frac{1}{p_1}+\frac{1}{p_2}+\cdots+\frac{1}{p_s}=\frac{1}{q_1}\left(\frac{1}{q_3}+\frac{1}{q_4}+\cdots+\frac{1}{q_{s+1}}\right)+\frac{1}{q_2}
    =\frac{1}{q_1}+\frac{1}{q_2}=1,
    \end{align*}
    and
    \begin{align*}
    p_1a_1<1,\ p_2a_2<1, \ldots,\ p_sa_s<1.
    \end{align*}
    For the case $a_1, a_2, \ldots,a_k\geq0$ satisfying $a_1+a_2+\cdots+a_k=1$. If $a_i>0$ for all $i\in\{1,\ldots,k\}$, then we can take $p_i=\frac{1}{a_i}$ such that $\frac{1}{p_1}+\frac{1}{p_2}+\cdots+\frac{1}{p_k}=1$ and (\ref{yibc}) holds. If $a_i=0$ for some $i\in\{1,\ldots,k\}$ which same as the case $(k-1)$.
    The proof is complete.
    \end{proof}

    By the above inequalities, we begin to establish the singular Trudinger-Moser type inequalities in $H^1_0(\Omega,\mathbb{R}^k)$:
    \begin{lemma}\label{lemtm2}
    ({\bfseries Improvement of the Trudinger-Moser inequality})
    Let $\Omega$ be a bounded domain in $\mathbb{R}^2$ containing the origin and $U=(u_1,\ldots,u_k)\in H^1_0(\Omega,\mathbb{R}^k)$. Then for every $\alpha >0$, and $\beta\in [0,2)$,
    \begin{align}\label{tmit}
    \int_{\Omega} \frac{e^{\alpha |U|^2}}{|x|^\beta}<\infty.
    \end{align}
    Moreover, it holds that
    \begin{align}\label{tmiht}
    \sup_{U\in H^1_0(\Omega,\mathbb{R}^k),\ \|U\|\leq 1}\int_{\Omega} \frac{e^{\alpha |U|^2}}{|x|^\beta}\leq C(\Omega).
    \end{align}
    if and only if $\frac{\alpha}{4\pi}+\frac{\beta}{2} \leq 1$, where $C(\Omega)$ is given in Lemma \ref{lemtm1}.
    \end{lemma}

    \begin{proof}
    Because $U=(u_1,\ldots,u_k)$, we can get $|U|^2=\sum^k_{i=1}|u_i|^2$. Thus, by using (\ref{yiy}) and Lemma \ref{lemtm1}, we have
    \begin{align*}
    \int_{\Omega} \frac{e^{\alpha |U|^2}}{|x|^\beta}=\int_{\Omega} \frac{e^{\alpha |u_1|^2}\cdots e^{\alpha |u_k|^2}}{|x|^\beta} \leq \sum^k_{i=1}\int_{\Omega} \frac{e^{k\alpha |u_i|^2}}{|x|^{\beta}}<\infty.
    \end{align*}
    Then we begin to proof (\ref{tmiht}). For each $U\in H^1_0(\Omega,\mathbb{R}^k)$ satisfying $\|U\|\leq 1$, then $\|U\|^2=\sum^k_{i=1}\|u_i\|^2\leq 1$, Lemma \ref{yib} shows that there exist $p_1,\ldots,p_k>1$ satisfying $\frac{1}{p_1}+\frac{1}{p_2}+\cdots+\frac{1}{p_k}=1$ such that
    $p_i\|u_i\|^2\leq 1$ holds, for all $i=1,2,\ldots,k$. If $\frac{\alpha}{4\pi}+\frac{\beta}{2} \leq 1$, then it also holds that $\frac{p_i\|u_i\|^2\alpha}{4\pi}+\frac{\beta}{2} \leq 1$, for all $i=1,2,\ldots,k$. Then by using Lemma \ref{yi}, i.e. the improvement of the Young's inequality, and from Lemma \ref{lemtm1} we have
    \begin{align*}
    \begin{split}
    \int_{\Omega} \frac{e^{\alpha |U|^2}}{|x|^\beta} =\int_{\Omega} \frac{e^{\alpha |u_1|^2}\cdots e^{\alpha |u_k|^2}}{|x|^\beta}  \leq \sum^k_{i=1}\int_{\Omega} \frac{e^{p_i\alpha |u_i|^2}}{p_i|x|^{\beta}} =\sum^k_{i=1}\int_{\Omega} \frac{e^{p_i\alpha\|u_i\|^2 (\frac{u_i}{\|u_i\|})^2}}{p_i|x|^{\beta}}
    \leq  \sum^k_{i=1}\frac{C(\Omega)}{p_i}=C(\Omega),
    \end{split}
    \end{align*}
    then (\ref{tmiht}) follows. %Lemma \ref{lemtm1} implies that the supremum for the above integral is not bigger that $C(\Omega)$.
    If $\frac{\alpha}{4\pi}+\frac{\beta}{2} >1$, we take $U=(u,0,\ldots,0)$, then Lemma \ref{lemtm1} shows that the supremum for the integral in (\ref{tmiht}) is infinite. Thus the proof is complete.
    \end{proof}

    \begin{lemma}\label{lemtm3}
    Let $\{U_n\}$ be a sequence of functions in $H^1_0(\Omega,\mathbb{R}^k)$ with $\|U_n\|=1$ such that $U_n\rightharpoonup U\neq0$ weakly in  $H^1_0(\Omega,\mathbb{R}^k)$. Then for any $0<p<\frac{2\pi(2-\beta)}{{(1-\|U\|^2)}}$ and $\beta\in [0,2)$, we have
    $$
    \sup_n \int_{\Omega}\frac{e^{p|U_n|^2}}{|x|^\beta}<\infty.
    $$
    \end{lemma}

    \begin{proof}
    Since $U_n\rightharpoonup U\neq0$ and $\|\nabla U_n\|_2=1$, we conclude that
    \begin{align*}
    \|U_n-U\|^2=1-2\langle U_n,U\rangle+\|U\|^2\rightarrow 1-\|U\|^2<\frac {2\pi(2-\beta)}{p}
    \end{align*}
    Thus, for large $n$ we have
    \begin{align*}
    \frac{p\|U_n-U\|^2}{4\pi}+\frac{\beta}{2}<1.
    \end{align*}
    Now we can choose $q>1$ close to 1 and $\epsilon>0$ such that
    \begin{align*}
    \frac{qp(1+\epsilon^2)\|U_n-U\|^2}{4\pi}+\frac{q\beta}{2}\leq 1.
    \end{align*}
    Lemmas \ref{lemtm2} shows that
    \begin{align*}
    \int_{\Omega}\frac{e^{qp(1+\epsilon^2)|U_n-U|^2}}{|x|^{q\beta}}\leq  C(\Omega).
    \end{align*}
    Moreover, since
    \begin{align*}
    p|U_n|^2 \leq p (1+ \epsilon ^2)|U_n-U|^2 + p(1+1/\epsilon ^2)|U|^2,
    \end{align*}
    which can be proved by Young inequality, then it follows that
    \begin{align*}
    e^{p|U_n|^2 }\leq e^{p (1+ \epsilon ^2)|U_n-U|^2} e^{ p(1+1/\epsilon ^2)|U|^2}.
    \end{align*}
    Consequently, by  H\"{o}lder inequality,
    \begin{align*}
    \begin{split}
    \int_{\Omega}\frac{e^{p|U_n|^2}}{|x|^\beta}&\leq \left (\int_{\Omega}\frac{e^{qp(1+\epsilon^2)|U_n-U|^2}}{|x|^{q\beta}}\right)^{1/q}\left(\int_{\Omega}e^{rp(1+1/\epsilon^2)|U|^2}\right)^{1/r}\leq C\left (\int_{\Omega}e^{rp(1+1/\epsilon^2)|U|^2}\right),
    \end{split}
    \end{align*}
    for large $n$, where $r=\frac{q}{q-1}$. By Lemma \ref{lemtm2}, we know the second term in the last inequality is bounded, and this lemma is proved.
    \end{proof}

\begin{remark}\label{remccp1}\rm
    Lemma \ref{lemtm3} is actually an expression that Concentration-compactness principle for a singular Trudinger-Moser inequality which is better than (\cite{ds}, Lemma 2.3). However, it is still not clear whether $\frac{2\pi(2-\beta)}{1-\|U\|^2}$ is sharp or not. This is still an open question.
    \end{remark}

    \begin{lemma}\label{lemtm4}
    If $V\in H^1_0(\Omega,\mathbb{R}^k),\ \alpha>0,\ q>0,\ \beta\in [0,2)$ and $\|V\|\leq N$ with $\frac{\alpha N^2}{4\pi}+\frac{\beta}{2}<1$, then there exists $C=C(\alpha,N,q)>0$ such that
    \begin{align}\label{asd1}
    \int_{\Omega} |V|^q \frac{e^{\alpha |V|^2}}{|x|^\beta}\leq C\|V\|^q.
    \end{align}
    \end{lemma}

    \begin{proof}
    We consider $r> 1$ close to 1 such that $\frac{r \alpha N^2}{4\pi}+\frac{r\beta}{2} \leq 1$ and $sq\geq 1$, where $s=\frac {r}{r-1}$. By using   H\"{o}lder  inequality and Lemma \ref{lemtm2}, we have
    \begin{align*}
    \begin{split}
    \int_{\Omega} |V|^q \frac{e^{\alpha |V|^2}}{|x|^\beta} &\leq  \left (\int_{\Omega}\frac{e^{r\alpha |V|^2}}{|x|^{r\beta}}\right)^{1/r}\|V\|^q_{qs}
    \leq C(\Omega)\|V\|^q_{qs}.
    \end{split}
    \end{align*}
    Finally, using the continuous embedding $H^1_0(\Omega,\mathbb{R}^k)\hookrightarrow L^{sq}(\Omega,\mathbb{R}^k)$, we conclude that
    \begin{align*}\label{asd1}
    \int_{\Omega} |V|^q \frac{e^{\beta |V|^2}}{|x|^\beta}\leq C\|V\|^q.
    \end{align*}
    \end{proof}

\section{{\bfseries The variational framework}}\label{vf}
    We now consider the functional $I$ given by
    \begin{equation*}
    I_\varepsilon(U)=\frac {1}{2}M(\|U\|^2)-\int_{\Omega}\frac{F(x,U)}{|x|^\beta}dx-\varepsilon\int_{\Omega}U\cdot H(x)dx.
    \end{equation*}

\begin{lemma}\label{ic1}
Under our assumptions we have that $I$ is well defined and $C^1$ on $H^1_0(\Omega,\mathbb{R}^k)$. Moreover,
\begin{equation*}
    \langle I_\varepsilon'(U),\Phi\rangle_*=m(\|U\|^2)\langle U,\Phi\rangle-\int_{\Omega}\frac{\Phi\cdot\nabla {F(x,U)}}{|x|^\beta} dx-\varepsilon\int_{\Omega}\Phi\cdot H dx,
    \end{equation*}
where $\Phi\in H^1_0(\Omega,\mathbb{R}^k)$, here $\langle \cdot,\cdot \rangle_*$ simply denotes the dual pairing between $H^1_0(\Omega,\mathbb{R}^k)$ and its dual space $\left(H^1_0(\Omega,\mathbb{R}^k)\right)^*$.
\end{lemma}

\begin{proof}
 We have that $f_i$ is continuous and has subcritical (or critical) growth at $\infty$, as defined in (\ref{1.2}) (or (\ref{1.3})). Thus, given $\alpha>0$ (or $\alpha>\alpha_0$), there exists $C>0$ such that $|f_i(x,U)|\leq Ce^{\alpha |U|^2}$ for all $(x,U)\in\Omega \times \mathbb{R}^k$. Then,
     \begin{equation} \label{3.1}
    |\nabla F(x,U)|\leq \sum^k_{i=1}|f_i(x,U)|\leq C_1 e^{\alpha |U|^2},
     \ \ \text{for all }(x,U)\in\Omega \times \mathbb{R}^k.\\[3pt]
    \end{equation}
By $(F_1)$, given $\epsilon>0$ there exists $\delta>0$ such that
    \begin{equation}\label{3.2}
    |F(x,U)|\leq \frac {\lambda_1 m_0-\epsilon}{2} |U|^2
     \ \ \text{always that }|U|<\delta.\\[3pt]
    \end{equation}
Thus, using (\ref{3.1}), (\ref{3.2}) and $(F_3)$, we have
    \begin{align*}
    \int_{\Omega}\frac{|F(x,U)|}{|x|^\beta}dx \leq \frac {\lambda_1 m_0-\epsilon}{2}\int_{\Omega}\frac{|U|^2}{|x|^\beta} dx + C_1\int_{\Omega}\frac{|U|e^{\alpha |U|^2}}{|x|^\beta}dx.
    \end{align*}
Considering the continuous imbedding $H^1_0(\Omega,\mathbb{R}^k)\hookrightarrow L^s(\Omega,\mathbb{R}^k)$ for $s\geq 1$ and using Lemma \ref{lemtm4}, it follows that $\frac{F(x,U)}{|x|^\beta}\in L^1(\Omega)$ which implies that $I_\varepsilon$ is well defined. And we can see that $I\in C^1(H^1_0(\Omega,\mathbb{R}^k),\mathbb R)$ with
    \begin{equation*}
    \langle I_\varepsilon'(U),\Phi\rangle_*=m(||U||^2)\langle U,\Phi\rangle-\int_{\Omega}\frac{\Phi\cdot\nabla {F(x,U)}}{|x|^\beta} dx-\varepsilon\int_{\Omega}\Phi\cdot H dx,
    \end{equation*}
for all $U, \Phi\in H^1_0(\Omega,\mathbb{R}^k)$.
\end{proof}

    From Lemma \ref{ic1}, we have that critical points of the functional $I_\varepsilon$ are precisely  weak solutions of problem (\ref{P}). In the next three lemmas we check that the functional $I_\varepsilon$ satisfies the geometric conditions of the Mountain-pass theorem.

    \begin{lemma}\label{lemgc1}
    Suppose that $(M_1)$ and $(F_1),\ (F_3)$ hold and the function $f_i$ has subcritical (or critical) growth at $\infty$. Then for small $\varepsilon$, there exist positive number $\rho_\varepsilon$ and $\varsigma$ such that
    \begin{align*}
    I_\varepsilon(U)\geq \varsigma,\ \ \forall U\in H^1_0(\Omega,\mathbb{R}^k)\ \ \mbox{with}\ \ \|U\|=\rho_\varepsilon.
    \end{align*}
    Moreover, $\rho_\varepsilon$ can be chosen such that $\rho_\varepsilon\to0$ as $\varepsilon\to0$.
    \end{lemma}

    \begin{proof}
     By $(F_1)$, given $\kappa>0$, there exists $\delta>0$ such that
    \begin{align*}
    |\nabla F(x,U)|\leq (\lambda_1 m_0-\kappa)|U|
    \end{align*}
    always that $|U|<\delta$. On the other hand, for $\alpha>0$ (subcritical case) or $\alpha>\alpha_0$ (critical case), we have that there exists $C_1>0$ such that $|f_i (x,U)|\leq C_1|U|^{q-1} e^{\alpha |U|^2}$ for all $|U|\geq \delta$ with $q>2$. Thus,
    \begin{align}\label{3.3}
    |\nabla F(x,U)|\leq \sum^k_{i=1} |f_i (x,U)|\leq (\lambda_1 m_0-\kappa)|U|+C_2|U|^{q-1} e^{\alpha |U|^2},\ \ \forall (x,U)\in \Omega\times \mathbb{R}^k.
    \end{align}
    Thus, by $(M_1),\ (F_3)$ and (\ref{3.3}),
    \begin{align*}
    \begin{split}
    I_\varepsilon(U)=&\frac {1}{2}M(\|U\|^2)-\int_{\Omega}\frac{F(x,U)}{|x|^\beta}dx-\varepsilon\int_{\Omega}U\cdot H(x)dx \\
     \geq &\frac{1}{2}m_0\|U\|^2-\int_{\Omega}\frac{F(x,U)}{|x|^\beta}dx-\varepsilon\int_{\Omega}U\cdot H(x)dx \\
     \geq &\frac{1}{2}m_0\|U\|^2-\frac {1}{\mu}\int_{\Omega}\frac{U\cdot \nabla F(x,U)}{|x|^\beta}dx-\varepsilon\|U\|\|H\|_* \\
     \geq &\frac{1}{2}m_0\|U\|^2-\frac {\lambda_1 m_0-\kappa}{\mu}\int_{\Omega}\frac{|U|^2}{|x|^\beta}dx-C_2\int_{\Omega}|U|^q \frac{e^{\alpha |U|^2}}{|x|^\beta}dx-\varepsilon\|U\|\|H\|_* \\
     \geq &\Big(\frac {m_0}{2}-\frac {m_0}{\mu}+\frac {\kappa}{\mu \lambda_1}\Big)\|U\|^2-C_2\int_{\Omega}|U|^q \frac{e^{\alpha |U|^2}}{|x|^\beta}dx-\varepsilon\|U\|\|H\|_*,
    \end{split}
    \end{align*}
    By Lemma \ref{lemtm4}, there exists $N>0$ such that $\alpha N^2/4\pi+\beta/2<1$ and we take $\|U\|\leq N$, there exists $C>0$ such that $\int_{\Omega}|U|^q \frac{e^{\alpha |U|^2}}{|x|^\beta}dx\leq C\|U\|^q$. Therefore,
    \begin{align*}
    I_\varepsilon(U)\geq \Big(\frac {m_0}{2}-\frac {m_0}{\mu}+\frac {\kappa}{\mu \lambda_1}\Big)\|U\|^2-C_3\|U\|^q-\varepsilon\|U\|\|H\|_*,
    \end{align*}
    Since $m_0>0,\ \mu>2\theta>2,\ q>2$, then for small $\varepsilon$, there exists $\xi_\varepsilon>0$ such that $\Big(\frac {m_0}{2}-\frac {m_0}{\mu}+\frac {\kappa}{\mu \lambda_1}\Big)\xi_\varepsilon^2-C_3\xi_\varepsilon^q-\varepsilon \xi_\varepsilon\|H\|_*>0$.
    Consequently, taking $\rho_\varepsilon=\min\{N,\xi_\varepsilon\}>0$, we can get that $I_\varepsilon(U)\geq\varsigma$ whenever $\|U\|=\rho_\varepsilon$ where $\varsigma:=\Big(\frac {m_0}{2}-\frac {m_0}{\mu}+\frac {\kappa}{\mu \lambda_1}\Big)\rho_\varepsilon^2-C_3\rho_\varepsilon^q-\varepsilon \rho_\varepsilon\|H\|_*>0$. And it is worth noting that, $\rho_\varepsilon\rightarrow 0$ as $\varepsilon\rightarrow 0$.
    \end{proof}

    \begin{lemma}\label{lemgc2}
    Assume that $(M_1),\ (M_3),\ (F_3)$ hold and the function $f_i$ has subcritical (or critical) growth at $\infty$. Then there exists $E\in H^1_0(\Omega,\mathbb{R}^k)$ with $\|E\|>\rho_\varepsilon$ such that
    \begin{align*}
     I_\varepsilon(E)<\inf_{\|U\|=\rho_\varepsilon}I(U).
    \end{align*}
    \end{lemma}

    \begin{proof}
     We shall make use of the polar coordinate representation
    \begin{align*}
    U=(\nu,\phi)=(\nu,\phi_1,\ldots,\phi_{k-1}),
    \end{align*}
    where $\nu\geq 1,\ -\pi\leq\phi_1\leq\pi,\ 0\leq\phi_2,\ldots,\phi_{k-1}\leq\pi$ and
    \begin{align*}
    \begin{split}
    u_1&=\nu \sin(\phi_1)\sin(\phi_2)\cdots\sin(\phi_{k-1}),\\
    u_2&=\nu \cos(\phi_1)\sin(\phi_2)\cdots\sin(\phi_{k-1}),\\
    u_3&=\nu \cos(\phi_2)\cdots\sin(\phi_{k-1}),\\
    \vdots \\
    u_k&=\nu \cos(\phi_{k-1}).
    \end{split}
    \end{align*}
    Substituting in $(F_3)$, we get $0<\mu F(x,U)\leq \nu F_{\nu}(x,U)$ and hence
    \begin{align*}
    F(x,U)\geq\left(\min_{|W|=1}F(x,W)\right)|U|^\mu,\ \ \mbox{for all}\ \ x\in \Omega \ \ \mbox{and}\ \ |U|\geq 1.
    \end{align*}
    Hence, for all $U\in H^1_0(\Omega,\mathbb{R}^k)\setminus\{0\}$ with $|U|\geq 1$, we have that
    \begin{align*}
    F(x,U)\geq C|U|^\mu.
    \end{align*}
    From (\ref{1.4}), we have that $0<M(t)\leq M(1) t^\theta$ for all $ t\geq 1$. Thus we have
    \begin{align*}
    I_\varepsilon(tU)\leq C_1t^{2\theta}\|U\|^{2\theta}-C_2 t^\mu \int_{K}\frac{|U|^\mu}{|x|^\beta} dx-t\varepsilon\int_{\Omega}U\cdot H dx,
    \end{align*}
    for $t$ large enough, where $C_1,\ C_2>0,\ \mu>2\theta>2$, $K$ is the compact subset of $\Omega$ which yields $I_\varepsilon(tU)\rightarrow -\infty$ as $t\rightarrow+\infty$. Setting $E=tU$ with $t$ large enough such that $I_\varepsilon(E)<0$ with $\|E\|>\rho_\varepsilon$. Thus, the proof is finished.
    \end{proof}

    \begin{lemma}\label{lemgc3}
    If $f_i$ has subcritical (or critical) growth at $\infty$, there exists $\eta_\varepsilon>0$ and $V\in H^1_0(\Omega,\mathbb{R}^k)\backslash \{\mathbf{0}\}$ such that $I_\varepsilon(tV)<0$ for all $0<t<\eta_\varepsilon$. In particular,
    \begin{align*}
    \inf_{\|U\|\leq \eta_\varepsilon}I_\varepsilon(U)<0.
    \end{align*}
    \end{lemma}

    \begin{proof}
    Since $h_i\in \big(\big(H^1_0(\Omega)\big)^*,\|\cdot\|_*\big)\backslash\{0\}$ for some $i\in\{1,\ldots,k\}$, then by the Riesz representation theorem, the problem
    \begin{align*}
    -\Delta v_i=h_i,\ \ x\in\Omega;\ \ v_i=0\ \ \mbox{on}\ \ \partial\Omega,
    \end{align*}
    has a unique nontrivial weak solution $v_i$ in $H^1_0(\Omega)$. Thus,
    \begin{align*}
    \int_{\Omega} h_i v_i=\|v_i\|^2>0.%,\ \ \mbox{because}\ \ \varepsilon>0,\ h_i\neq 0,\ i=1,2,\ldots,k.
    \end{align*}
    Since $f_i(x,0,\ldots,0)=0$, by continuity, $(F_3)$ and (\ref{1.4}),  $\theta>1$, it follows that there exists $\eta_\varepsilon>0$ such that for all $0<t<\eta_\varepsilon$,
    \begin{align*}
    \begin{split}
    \frac{d}{dt}[I_\varepsilon((0,\ldots,tv_i,\ldots,0))]=&m(t^2\|v_i\|^2)t\|v_i\|^2-\int_{\Omega}\frac{v_if_i(x,(0,\ldots,tv_i,\ldots,0))}{|x|^\beta}
    -\varepsilon\int_{\Omega} h_i v_i \\
    \leq &C\|v_i\|^{2}t-\varepsilon\|v_i\|^2-\int_{\Omega}\frac{v_if_i(x,(0,\ldots,tv_i,\ldots,0))}{|x|^\beta}
    <0.
    \end{split}
    \end{align*}
    Using that $I_\varepsilon(\mathbf{0})=0$, it must hold that $I_\varepsilon(tV)<0$ for all $0<t<\eta_\varepsilon$ where $V=(0,\ldots,v_i,\ldots,0)$.
    \end{proof}

\section{{\bfseries On Palais-Smale sequences}}\label{ps}
    To prove that a Palais-Smale sequence converges to a weak solution of problem (\ref{P}), we need to establish the following lemma.

    \begin{lemma}\label{lem4.1}
    Assume that $(M_1),\ (M_3),\ (F_3)$ hold and $f_i$ has subcritical (or critical) growth at $\infty$. Let $\{U_n\}\subset H^1_0(\Omega,\mathbb{R}^k)$ be the Palais-Smale sequence for functional $I_\varepsilon$ at finite level. Then there exist $C>0$ such that
    \begin{align*}
    \|U_n\|\leq C,\ \ \int_{\Omega}\frac{\left|U_n\cdot \nabla F(x,U_n)\right|}{|x|^\beta}dx\leq C\ \ \mbox{and}\ \ \int_{\Omega}\frac{F(x,U_n)}{|x|^\beta}dx\leq C.
    \end{align*}
    \end{lemma}

    \begin{proof}
    Let $\{U_n\} \subset H^1_0(\Omega,\mathbb{R}^k)$ be a sequence such that $I_\varepsilon(U_n)\rightarrow c$ and $I_\varepsilon'(U_n)\rightarrow 0$, where $|c|<\infty$, then we can take this as follows:
    \begin{equation}\label{4.1}
    I_\varepsilon(U_n)=\frac{1}{2}M(\|U_n\|^2)-\int_{\Omega}\frac{F(x,U_n)}{|x|^\beta}dx-\varepsilon\int_{\Omega}U_n\cdot H dx=c+\delta_n,
    \end{equation}
    where $\delta_n\rightarrow 0$ as $n\rightarrow \infty$, and
    \begin{equation}\label{4.2}
    \langle I_\varepsilon'(U_n),U_n\rangle_*=m(\|U_n\|^2)\|U_n\|^2-\int_{\Omega}\frac{U_n\cdot \nabla F(x,U_n)}{|x|^\beta}dx-\varepsilon\int_{\Omega}U_n\cdot H dx=o(\|U_n\|).
    \end{equation}
    Then for $n$ large enough, by $(M_1),\ (F_3)$ and (\ref{1.4}), it holds that
    \begin{align*}
    \begin{split}
    C+\|U_n\|\geq& I_\varepsilon(U_n)-\frac {1}{\mu} \langle I'_\varepsilon(U_n),U_n\rangle_* \\
    =&\frac {1}{2}M(\|U_n\|^2)-\frac {1}{\mu}m(\|U_n\|^2)\|U_n\|^2+\frac{1}{\mu} \int_{\Omega}\frac{\left (U_n\cdot \nabla F(x,U_n) -\mu F(x,U_n)\right )}{|x|^\beta}dx \\
    &-\frac{\mu-1}{\mu}\varepsilon\int_{\Omega}U_n\cdot H dx \\
    \geq& \frac {1}{2\theta}\Big[\theta M(\|U_n\|^2)-m(\|U_n\|^2)\|U_n\|^2\Big]+\Big(\frac {1}{2\theta}-\frac {1}{\mu}\Big)m(\|U_n\|^2)\|U_n\|^2 \\
    &-\frac{\mu-1}{\mu}\varepsilon \|U_n\| \|H\|_* \\
    \geq& \Big(\frac {1}{2\theta}-\frac {1}{\mu}\Big)m(\|U_n\|^2)\|U_n\|^2-\frac{\mu-1}{\mu}\varepsilon \|U_n\| \|H\|_* \\
    \geq& \frac {\mu-2\theta}{2\theta\mu}m_0\|U_n\|^2-\frac{\mu-1}{\mu}\varepsilon \|U_n\| \|H\|_*,
    \end{split}
    \end{align*}
    for some $C>0$. Since $\mu>2\theta>2,\ m_0>0,\ \varepsilon>0$, we obtain $\|U_n\|$ is bounded. From (\ref{4.1}) and (\ref{4.2}), it can be concluded directly that there exist $C>0$ such that $\int_{\Omega}\frac{U_n\cdot \nabla F(x,U_n)}{|x|^\beta}dx\leq C$ and $\int_{\Omega}\frac{F(x,U_n)}{|x|^\beta}dx\leq C$. Condition $(F_3)$ implies $U_n\cdot \nabla F(x,U_n)\geq 0$ for all $x\in\Omega$, thus we have $\int_{\Omega}\frac{\left|U_n\cdot \nabla F(x,U_n)\right|}{|x|^\beta}dx\leq C$.
    \end{proof}

    In order to show that the limit of a sequence in $H^1_0(\Omega,\mathbb{R}^k)$ is a weak solution of problem (\ref{P}) we will use the following convergence result due to Figueiredo-do \'{O}-Ref \cite{ddr2} and the dominated result due to do \'{O}-Medeiros-Severo \cite{dms}.
    \begin{lemma}\label{lemcr}
    \cite{ddr2} Let $\Omega \in \mathbb{R}^2$ be a bounded domain and $f: \Omega\times \mathbb{R} \rightarrow \mathbb{R}$ be a continuous function, $\beta \in [0,2)$. Then for any sequence $\{u_n\}$ in $L^1(\Omega)$ such that
    \begin{align*}
    u_n\rightarrow u\ \ in\ \ L^1(\Omega),\ \ \frac{f(x,u_n)}{|x|^\beta} \in L^1(\Omega)\ \ \mbox{and}\ \ \int_{\Omega}\frac{|f(x,u_n)u_n|}{|x|^\beta}dx\leq C,
    \end{align*}
    up to a sequence we have
    \begin{align*}
    \frac{f(x,u_n)}{|x|^\beta}\rightarrow \frac{f(x,u)}{|x|^\beta} \ \ \mbox{in}\ \ L^1(\Omega).
    \end{align*}
    \end{lemma}

    \begin{lemma}\label{lemdr}
    \cite{dms} Let $\{u_n\}$ be a sequence of functions in $H^1_0(\Omega)$ strongly convergent. Then there exists a subsequence $\{u_{n_k}\}$ of $\{u_n\}$ and $g\in H^1_0(\Omega)$ such that $|u_{n_k}(x)|\leq g(x)$ almost everywhere in $\Omega$.
    \end{lemma}

    \begin{lemma}\label{lem4.4}
    Assume $(M_1),\ (M_3),\ (F_2),\ (F_3)$ hold and $f_i$ has subcritical (or critical) growth at $\infty$. Let $\{U_n\}\subset H^1_0(\Omega,\mathbb{R}^k)$ be the Palais-Smale sequence for functional $I_\varepsilon$ at finite level, then there exists $U\in H^1_0(\Omega,\mathbb{R}^k)$ such that
    \begin{equation}\label{dcfi}
    \frac{f_i(x,U_n)}{|x|^\beta}\rightarrow \frac{f_i(x,U)}{|x|^\beta}\ \ \mbox{in}\ \ L^1(\Omega), \ \ \mbox{for all}\ \ i=1,\ldots,k.
    \end{equation}
    and
    \begin{equation}\label{dch}
    U_n\cdot H\rightarrow U\cdot H\ \ \mbox{in}\ \ L^1(\Omega).
    \end{equation}
    Moreover,
    \begin{equation}\label{4.5}
    \frac{F(x,U_n)}{|x|^\beta}\rightarrow \frac{F(x,U)}{|x|^\beta}\ \ \mbox{in}\ \ L^1(\Omega).
    \end{equation}
    \end{lemma}

    \begin{proof}
    According to Lemma \ref{lem4.1}, we know that $\{U_n\}$ is bounded in $H^1_0(\Omega,\mathbb{R}^k)$, then up to a subsequence, for some $U \in H^1_0(\Omega,\mathbb{R}^k)$ such that $U_n\rightharpoonup U$ weakly in $H^1_0(\Omega,\mathbb{R}^k)$,\ $U_n\rightarrow U$ in $L^p(\Omega,\mathbb{R}^k)$ for all $p\geq 1$ and $U_n(x)\rightarrow U(x)$ almost everywhere in $\Omega$. Consequently, by Lemmas \ref{lem4.1} and \ref{lemcr}, we have
    \begin{equation*}
    \frac{f_i(x,U_n)}{|x|^\beta}\rightarrow \frac{f_i(x,U)}{|x|^\beta}\ \ \mbox{in}\ \ L^1(\Omega), \ \ \mbox{for all}\ \ i=1,\ldots,k.
    \end{equation*}
    Since
    \begin{align*}
    \int_{\Omega}|U_n\cdot H-U\cdot H|dx\leq \int_{\Omega}|H| |U_n -U|dx\leq \|H\|_2\|U_n-U\|_2 \rightarrow 0,
    \end{align*}
    we can conclude that
    \begin{equation*}
    U_n\cdot H\rightarrow U\cdot H\ \ \mbox{in}\ \ L^1(\Omega).
    \end{equation*}
    Then by Lemma \ref{lemdr}, there exists $g_i\in L^1(\Omega)$ such that $\frac{|f_i(x,U_n)|}{|x|^\beta}\leq g_i$ almost everywhere in $\Omega$. From $(F_2)$ we can conclude that
    \begin{align*}
    |F(x,U_n)|\leq \sup_{\Omega \times [-S_0,S_0]} |F(x,U_n(x))|+M_0|\nabla F(x,U_n)| \ \ \mbox{a.e. in}\ \ \Omega.
    \end{align*}
    Thus, by the generalized Lebesgue dominated convergence theorem, we get
    \begin{equation*}
    \frac{F(x,U_n)}{|x|^\beta}\rightarrow \frac{F(x,U)}{|x|^\beta}\ \ \mbox{in}\ \ L^1(\Omega).
    \end{equation*}
    \end{proof}

\section{{\bfseries Proof of the main results}}\label{main}

    In order to obtain a weak solution with positive energy, according to Lemmas \ref{lemgc1} and \ref{lemgc2}, let
    \begin{equation}\label{defmpl}
    c_{M,\varepsilon}=\inf_{\gamma\in\Upsilon}\max_{t\in [0,1]}I_\varepsilon(\gamma(t))>0,
    \end{equation}
    be the minimax level of $I_\varepsilon$, where $\Upsilon=\{\gamma\in C\big([0,1],H^1_0(\Omega,\mathbb{R}^k)\big):\gamma(0)=\mathbf{0}, I_\varepsilon(\gamma(1))<0\}$. Therefore, using the Mountain-pass theorem, there exists a sequence $\{U_n\} \subset H^1_0(\Omega,\mathbb{R}^k)$ satisfying
    \begin{equation}\label{5.1}
    I_\varepsilon(U_n)\rightarrow c_{M,\varepsilon}\ \ \mbox{and}\ \ I'_\varepsilon(U_n)\rightarrow 0.
    \end{equation}
    And in order to obtain another weak solution with negative energy, by Lemmas \ref{lemgc1} and \ref{lemgc3}, we take $\eta_\varepsilon\leq \rho_\varepsilon$ and so we have that
    \begin{equation}\label{defc0}
    -\infty<c_{0,\varepsilon}:=\inf_{\|V\|\leq \rho_\varepsilon}I_\varepsilon(V)<0,
    \end{equation}
    where $\rho_\varepsilon$ is given as in Lemma \ref{lemgc1}. Since $\overline{B}_{\rho_\varepsilon}$ is a complete metric space with the metric given by the norm of $H^1_0(\Omega,\mathbb{R}^k)$, convex and the functional $I_\varepsilon$ is of class $C^1$ and bounded below on $\overline{B}_{\rho_\varepsilon}$, by the Ekeland variational principle, there exists a sequence $\{V_n\}$ in $\overline{B}_{\rho_\varepsilon}$ such that
    \begin{equation}\label{5.2}
    I_\varepsilon(V_n)\rightarrow c_{0,\varepsilon}\ \ \mbox{and}\ \ I'_\varepsilon(V_n)\rightarrow 0.
    \end{equation}

\medskip

%\noindent{\bfseries5.1 Subcritical case} ($f_i$ has subcritical growth at $\infty$)

\subsection{Subcritical case: Proof of Theorem \ref{thm1.2}}\label{ssectpfthmsc}

    In this subsection, we assume that $f_i$ has subcritical growth at $\infty$ satisfying (\ref{1.2}) and proof Theorem \ref{thm1.2}.

    \begin{lemma}\label{lem5.1}
    The functional $I_\varepsilon$ satisfies the Palais-Smale condition at any finite level $c$.
    \end{lemma}

    \begin{proof}
    Let $\{U_n\} \subset H^1_0(\Omega,\mathbb{R}^k)$ be a sequence such that $I_\varepsilon(U_n)\rightarrow c$ and $I_\varepsilon'(U_n)\rightarrow 0$. Lemma \ref{lem4.1} shows that $\{U_n\}$ is bounded in $H^1_0(\Omega,\mathbb{R}^k)$, then we can get a subsequence still labeled by $\{U_n\}$, for some $U\in H^1_0(\Omega,\mathbb{R}^k)$ such that
    \begin{align*}
   U_n\rightharpoonup U\ \ \mbox{in}\ \ H^1_0(\Omega,\mathbb{R}^k);\quad U_n \rightarrow U\ \ \mbox{in}\ \ L^q(\Omega,\mathbb{R}^k)   \ \mbox{for all}\ \ q\geq 1.
    \end{align*}
    Since
    \begin{equation}\label{5.3}
    \begin{split}
    \langle I_\varepsilon'(U_n),U_n-U\rangle_*=&m(\|U_n\|^2)\langle U_n,U_n-U\rangle-\int_{\Omega}\frac{(U_n-U)\cdot\nabla F(x,U_n)}{|x|^\beta}dx
    -\varepsilon\int_{\Omega}(U_n-U_0)\cdot H dx.
    \end{split}
    \end{equation}
    From $I_\varepsilon'(U_n)\rightarrow 0$ in $\big(H^1_0(\Omega,\mathbb{R}^k)\big)_*$, we have $\langle I'_\varepsilon(U_n),U_n-U\rangle_* \rightarrow 0$.
    Meanwhile, Lemma \ref{lem4.1} shows that $\|U_n\|$ is bounded, i.e. $\|U_n\|^2\leq C_0$ for some $C_0>0$, then by subscritical condition, the H\"{o}lder's inequality and Lemma \ref{yi}, it follows that
    \begin{align*}
    \begin{split}
    \left|\int_{\Omega}\frac{(U_n-U)\cdot\nabla F(x,U_n)}{|x|^\beta}dx \right| & \leq \int_{\Omega}\frac{|U_n-U| |\nabla F(x,U_n)|}{|x|^\beta}dx \\
    &\leq C_1\int_{\Omega}|U_n-U|\frac{e^{\alpha |U_n|^2}}{|x|^\beta}dx \\
    &\leq C_2\|U_n-U\|_{\frac {r}{r-1}}\left(\int_{\Omega}\frac{e^{r\alpha |U_n|^2}}{|x|^{r\beta}}dx\right)^\frac {1}{r} \\
    &\leq C_2\|U_n-U\|_{\frac {r}{r-1}}\left(\sum^k_{i=1}\int_{\Omega}\frac{e^{kr\alpha\|U_n\|^2\left(\frac{u^i_n}{\|U_n\|}\right)^2}}{|x|^{r\beta}}dx\right)^\frac {1}{r}
    \end{split}
    \end{align*}
    for some $C_1,C_2>0$, where $\alpha=\frac{4\pi(1-r\beta/2)}{krC_0}$ and $r>1$ sufficiently close to 1 such that $r\beta<2$. Since $\frac{kr\alpha\|U_n\|^2}{4\pi}+\frac{r\beta}{2}\leq 1$, then by Lemma \ref{lemtm1} and $H^1_0(\Omega,\mathbb{R}^k)\hookrightarrow L^s(\Omega,\mathbb{R}^k)$ is compact for $s\geq 1$, the fourth term in the last inequality converges to zero.
    Thus from (\ref{dch}), in (\ref{5.3}), it must be that
    \begin{align*}
    m(\|U_n\|^2)\langle U_n,U_n-U\rangle \rightarrow 0\ \ \mbox{as}\ \ n\rightarrow+\infty.
    \end{align*}
    Because $(M_1)$: $m(t)\geq m_0>0$ for $t\geq 0$ and $U_n\rightharpoonup U$ in $H^1_0(\Omega,\mathbb{R}^k)$, it must be that
    \begin{align*}
    \langle U_n,U_n-U\rangle\rightarrow 0,
    \end{align*}
    which means $\|U_n\|^2\rightarrow \|U\|^2$. By Radon's Theorem, $U_n\rightarrow U$ strongly in $H^1_0(\Omega,\mathbb{R}^k)$. This proof is complete.
    \end{proof}

\noindent{\bfseries Proof of Theorem \ref{thm1.2}.}
    By (\ref{5.1}), (\ref{5.2}) and Lemma \ref{lem5.1}, there exists $\varepsilon_{sc}>0$ such that for each $0<\varepsilon<\varepsilon_{sc}$, using Minimax principle, there exist critical points $U_{M,\varepsilon}$ for $I_\varepsilon$ at level $c_{M,\varepsilon}$ and $V_{0,\varepsilon}$ for $I_\varepsilon$ at level $c_{0,\varepsilon}$. We claim that $U_{M,\varepsilon}\neq \mathbf{0}$. In fact, suppose by contradiction that $U_{M,\varepsilon}\equiv \mathbf{0}$. We can know that $0<c_{M,\varepsilon}=\lim_{n\rightarrow\infty}I_\varepsilon(U_n)=I_\varepsilon(U_{M,\varepsilon})=I(\mathbf{0})=0$, what is absurd. Similarly, we have $V_{0,\varepsilon}\neq \mathbf{0}$. In the end, we claim $U_{M,\varepsilon}\neq V_{0,\varepsilon}$. Suppose by contradiction that $ U_{M,\varepsilon}\equiv V_{0,\varepsilon}$, then $0>c_{0,\varepsilon}=\lim_{n\rightarrow\infty}I_\varepsilon(V_n)=I(V_{0,\varepsilon})=I(U_{M,\varepsilon})=\lim_{n\rightarrow\infty}I_\varepsilon(U_n)=c_{M,\varepsilon}>0$, what is absurd. Thus, the proof of Theorem \ref{thm1.2} is complete.
    \qed

\medskip

%\noindent{\bfseries5.2 Critical case} ($f_i$ has critical growth at $\infty$)

\subsection{Critical case: Proof of Theorem \ref{thm1.3}}\label{ssectpfthmc}

    In this subsection, we assume that $f_i$ has critical growth at $\infty$ satisfying (\ref{1.3}) and give the proof Theorem \ref{thm1.3}.

    Firstly, we give a conclusion that functional $I_\varepsilon$ satisfies the Palais-Smale condition if Palais-Smale sequence less than appropriate level:

    \begin{lemma}\label{ms1}
    If $\{V_n\}$ is a Palais-Smale sequence for $I_\varepsilon$ at any finite level with
    \begin{align}\label{msVx}
    \liminf_{n\rightarrow\infty}\|V_n\|^2<\frac{2\pi(2-\beta)}{\alpha_0},
    \end{align}
    then $\{V_n\}$ possesses a strongly subsequence.
    \end{lemma}

    \begin{proof}
    Let $\{V_n\}\subset H^1_0(\Omega,\mathbb{R}^k)$ such that $I_\varepsilon(V_n)\rightarrow c$ and $I'_\varepsilon(V_n)\rightarrow 0$ in $\big(H^1_0(\Omega,\mathbb{R}^k)\big)^*$. By Lemma \ref{lem4.1}, $\|V_n\|\leq C$ for some $C>0$, thus, up to a sequence, for some $V\in H^1_0(\Omega,\mathbb{R}^k)$
    \begin{align*}
    \begin{split}
    &V_n\rightharpoonup V\ \ \mbox{in}\ \ H^1_0(\Omega,\mathbb{R}^k).%,\\&V_n \rightarrow V\ \ \mbox{in}\ \ L^s(\Omega,\mathbb{R}^k),\ \ \mbox{for all}\ \ s\geq 1.
    \end{split}
    \end{align*}
    Taking $V_n=V+W_n$, it follows that $W_n\rightharpoonup 0$ in $H^1_0(\Omega,\mathbb{R}^k)$ and by the Br\'{e}zis-Lieb Lemma (see \cite{bl2}), we get
    \begin{align}\label{ws1}
    \begin{split}
    \|V_n\|^2=\|V\|^2+\|W_n\|^2+o_n(1).
    \end{split}
    \end{align}
    By $V_n\rightharpoonup V$ which means $\langle V_n,V \rangle\rightarrow \langle V,V \rangle=\|V\|^2$. Therefore, (\ref{ws1}) can be replaced by
    \begin{align}\label{ws2}
    \begin{split}
    \|V_n\|^2=\langle V_n,V\rangle+\|W_n\|^2+o_n(1).
    \end{split}
    \end{align}
    By $I_\varepsilon'(V_n)\rightarrow 0$ in $\big(H^1_0(\Omega,\mathbb{R}^k)\big)^*$ and (\ref{dch}), (\ref{ws2}), we can get
    \begin{align*}
    \begin{split}
    \langle I_\varepsilon'(V_n),V_n-V\rangle_* &=m(\|V_n\|^2)\langle V_n,V_n-V \rangle-\int_{\Omega}\frac{(V_n-V)\cdot\nabla F(x,V_n)}{|x|^\beta}+o_n(1) \\
    &=m(\|V_n\|^2)\|W_n\|^2-\int_{\Omega}\frac{W_n\cdot\nabla F(x,V_n)}{|x|^\beta}+o_n(1),
    \end{split}
    \end{align*}
    that is,
    \begin{align}\label{ws3}
    \begin{split}
    m(\|V_n\|^2)\|W_n\|^2=\int_{\Omega}\frac{W_n\cdot\nabla F(x,V_n)}{|x|^\beta}+o_n(1).
    \end{split}
    \end{align}
    From (\ref{msVx}), there exists $\zeta>0$ such that $\alpha_0\|V_n\|^2<\zeta<2\pi(2-\beta)$ for $n$ sufficiently large and also, there exist $\alpha>\alpha_0$ close to $\alpha_0$ and $q>q$ close to~1~such that $q\alpha \|V_n\|^2<\zeta<2\pi(2-q\beta)$  for $n$ sufficiently large. Then by (\ref{3.1}), we have
    \begin{align*}
    \begin{split}
    \left|\int_{\Omega}\frac{W_n\cdot\nabla F(x,V_n)}{|x|^\beta}\right|\leq C_1\int_{\Omega}|W_n|\frac{e^{\alpha|V_n|^2}}{|x|^\beta},
    \end{split}
    \end{align*}
    and by the H\"{o}lder's inequality and  Lemma \ref{lemtm2}, we can get that
    \begin{align*}
    \begin{split}
    \int_{\Omega}|W_n|\frac{e^{\alpha|V_n|^2}}{|x|^\beta}\leq C_1\|W_n\|_s \bigg(\int_{\Omega}\frac{e^{q\alpha|V_n|^2}}{|x|^{q\beta}}\bigg)^{1/r}\leq C_2\|W_n\|_s,
    \end{split}
    \end{align*}
    where $s=\frac{q}{q-1}$. By the compact embedding $H^1_0(\Omega,\mathbb{R}^k)\hookrightarrow L^s(\Omega,\mathbb{R}^k)$ for $s\geq 1$, we conclude that
    \begin{align*}
    \begin{split}
    \int_{\Omega}\frac{W_n\cdot\nabla F(x,V_n)}{|x|^\beta}\rightarrow 0.
    \end{split}
    \end{align*}
   % By $(M_1)$: $m(t)\geq m_0>0$ for all $t\geq 0$,
    Thus, this together with (\ref{ws3}) and $(M_1)$,  we get that $\|W_n\|\rightarrow 0$ and the result follows.
    \end{proof}

    Then, in order to get a more precise information about the minimax level $c_{M,\varepsilon}$, let us consider the following sequence which was introduced in \cite{ddr}: for $n\in \mathbb{N}$ set $\delta_n=\frac{2\log n}{n}$, and let
    \begin{eqnarray*}
    y_n(t)=\frac {1}{\sqrt{2\pi}}
    \left\{ \arraycolsep=1.5pt
       \begin{array}{ll}
        \frac {t}{n^{1/2}}(1-\delta_n)^{1/2},\ \ \ \ \ \ \ \ \ \ \ \ \ \ \ \ \ \ \ \ \ \ \ \ \ \ \ \ \ \ \ \ \ \ \ \ \ \ \ \ \ &{\rm if}\ \ 0\leq t\leq n,\\[2mm]
        \frac {1}{\big [n(1-\delta_n)\big ]^{1/2}}\log\frac{A_n+1}{A_n+e^{-(t-n)}}+\big [n(1-\delta_n)\big ]^{1/2},\ \ \ &{\rm if}\ \ \ t\geq n,\\[2mm]
        \end{array}
    \right.
    \end{eqnarray*}
    where $A_n$ is defined as $A_n=\frac{1}{en^2}+O(\frac{1}{n^4})$.
    The sequence of function $\{y_n\}$ satisfies the following properties:
    \begin{eqnarray*}
    \left\{ \arraycolsep=1.5pt
       \begin{array}{ll}
        \{y_n\}\subset C\big([0,+\infty)\big),\ \mbox{piecewise differentiable, with}\ y_n(0)=0 \ \mbox{and}\  y'_n(t)\geq 0;\\[2mm]
        \int^{+\infty}_0 |y'_n(t)|^2=1;\\[2mm]
        \lim_{n\rightarrow \infty}\int^{+\infty}_0 e^{y^2_n(t)-t}dt=1+e.
        \end{array}
    \right.
    \end{eqnarray*}
    Now, let $y_n(t)=2\sqrt{\pi}\widehat{G}_n(e^{-t/2})$ with $|x|=e^{-t/2}$, define a function $\widehat{G}_n(x)=\widehat{G}_n(|x|)$ on $\overline{B_1(0)}$, which is nonnegative and radially symmetric. Moreover, we have
    \begin{align*}
    \int_{B_1(0)}|\nabla \widehat{G}_n(x)|^2dx=\int^{+\infty}_0|y'_n(t)|^2=1.
    \end{align*}
    Therefore $\|\widehat{G}_n\|=1$. Let $\tau=\frac{2-\beta}{2}$, then $\widehat{G}_n$ defines another function nonnegative and radially symmetric $\tilde{G}_n$ as follows:
    \begin{align*}
    \widehat{G}_n(\varrho)=\tau^{1/2}\tilde{G}_n(\varrho^{1/\tau})\ \ \mbox{for}\ \ \varrho\in [0,1].
    \end{align*}
    Note that
    \begin{align*}
    \int^1_0|\widehat{G}'_n(\varrho)|^2\varrho d\varrho=\int^1_0|\tilde{G}'_n(\varrho)|^2\varrho d\varrho.
    \end{align*}
    Therefore, $\|\widehat{G}_n\|=\|\tilde{G}_n\|$. The open ball $B_d(0)$ is contained in $\Omega$, where $d$ was given in $(F_4)$. Considering
    \begin{align}\label{defgnd}
    \mathscr{G}_{n,d}(x):=\big(G_{n,d}(x),0,\ldots,0\big),\ \ \ \mbox{where}\ \ \ G_{n,d}(x):=\tilde{G}_n\left(\frac{x}{d}\right),
    \end{align}
    then $\mathscr{G}_{n,d}(x)$ belongs to $H^1_0(\Omega,\mathbb{R}^k)$ with $\|\mathscr{G}_{n,d}\|=1$, and the support of $\mathscr{G}_{n,d}$ contained in $B_d(0)$.

    \begin{remark}\label{rem5.2} \rm
    If condition $(F_4)$ holds, we define $\mathscr{G}'_{n,d}(x)$ that $i$-th component is set to $G_{n,d}(x)$, and the remaining components are set to $0$, i.e., $\mathscr{G}'_{n,d}(x)=(0,\ldots,0,G_{n,d}(x),0,\ldots,0)$, then given $\delta>0$ there exists $s_\delta>0$ such that
    \begin{align*}
    \begin{split}
    \mathscr{G}'_{n,d}\cdot\nabla F(x,\mathscr{G}'_{n,d})
    =   G_{n,d} f_i(x,\mathscr{G}'_{n,d})
    \geq   (\eta_0 -\delta)\exp\left(\alpha_0|\mathscr{G}'_{n,d}|^2\right)
    =   (\eta_0 -\delta)\exp\left(\alpha_0|G_{n,d}|^2\right),
    \end{split}
    \end{align*}
    $\forall x\in \Omega,\ \ |\mathscr{G}'_{n,d}|=|G_{n,d}|\geq s_\delta$. This is the same as type (\ref{5.7}) below,
    therefore without lose of generality, we can assume that $i=1$ in $(F_4)$.
    \end{remark}

    \begin{lemma}\label{leest}
    For any $0<\epsilon<1$, we have that for $x\in B_\frac{d}{\varpi(n)}(0)$ with $\varpi(n)=\exp\left\{\frac{n^{(1+\epsilon)/2}}{2}\right\}$,
     \begin{align}\label{psbc}
     |\mathscr{G}_{n,d}(x)|\geq\frac{1}{2\sqrt{\pi}} n^{\frac{\epsilon}{2}}\left(1-\frac{2\log n}{n}\right)^{\frac{1}{2}},%\qquad \mbox{with}\ \ \varpi=\exp\left\{\frac{n^{(1+\epsilon)/2}}{2}\right\}.
    \end{align}
    where $\mathscr{G}_{n,d}$ is given in (\ref{defgnd}).
    \end{lemma}
    \begin{proof}
    For $x\in B_\frac{d}{\varpi(n)}(0)$,% where $\varpi$ is given in (\ref{psbc}),
    \begin{align*}
    \left|\mathscr{G}_{n,d}(x)\right|=\left|G_{n,d}(x)\right|=\left|\tilde{G}_n\left(\frac{x}{d}\right)\right|=\left|\tilde{G}_n(y)\right|,
    \end{align*}
    where $y=\frac{x}{d}\in B_\frac{1}{\varpi(n)}(0)$. Moreover,
    \begin{align*}
    \left|\tilde{G}_n(y)\right|=\left|\tilde{G}_n(|y|)\right|=\frac{1}{2\sqrt{\pi}}y_n(-2\log(|y|))=\frac{1}{2\sqrt{\pi}}y_n(t),
    \end{align*}
    where $t=-2\log(|y|)\in (n^{\frac{1+\epsilon}{2}},+\infty)$. Noticing that, $y_n(t)\geq\big[n(1-\delta_n)\big]^{\frac{1}{2}}=n^{\frac{1}{2}}\left(1-\frac{2\log n}{n}\right)^{\frac{1}{2}}$ if $t\geq n$. Moreover, in $(n^{\frac{1+\epsilon}{2}},n)$,
    \begin{align*}
    \begin{split}
    y_n(t)\geq \frac {n^{\frac{1+\epsilon}{2}}}{n^{\frac{1}{2}}}(1-\delta_n)^{\frac{1}{2}}=n^{\frac{\epsilon}{2}}\left(1-\frac{2\log n}{n}\right)^{\frac{1}{2}}.
    \end{split}
    \end{align*}
    The proof is complete.
   \end{proof}

    \begin{lemma}\label{nl}
    If conditions $(M_1),\ (M_3)$ and $(F_3),\ (F_4)$ hold, then
    \begin{align*}
    \max_{t\geq 0}\left[\frac{1}{2}M(t^2)-\int_{\Omega}\frac{F(x,t\mathscr{G}_{n,d})}{|x|^\beta}\right]
    <\frac{1}{2}M\left(\frac{2\pi(2-\beta)}{\alpha_0}\right).
    \end{align*}
    \end{lemma}

    \begin{proof}
    Suppose by contradiction, that for all $n\in \mathbb{N}$, we have
    \begin{align*}
    \max_{t\geq 0}\left[\frac{1}{2}M(t^2)-\int_{\Omega}\frac{F(x,t\mathscr{G}_{n,d})}{|x|^\beta}\right]
    \geq \frac{1}{2}M\left(\frac{2\pi(2-\beta)}{\alpha_0}\right).
    \end{align*}
    By Lemmas \ref{lemgc1} and \ref{lemgc2}, for each $n$ there exists $t_n>0$ such that
    \begin{align*}
    \frac{1}{2}M(t^2_n)-\int_{\Omega}\frac{F(x,t_n\mathscr{G}_{n,d})}{|x|^\beta}=\max_{t\geq 0}\left[\frac{1}{2}M(t^2)-\int_{\Omega}\frac{F(x,t\mathscr{G}_{n,d})}{|x|^\beta}\right]
    \end{align*}
    From this, and using $(F_3)$, one has $M(t^2_n)\geq M\left(\frac{2\pi(2-\beta)}{\alpha_0}\right)$. By $(M_1)$, which implies that $M:[0,+\infty)\rightarrow[0,+\infty)$ is an increasing bijection and so
    \begin{align}\label{5.5}
    t^2_n\geq \frac{2\pi(2-\beta)}{\alpha_0}.
    \end{align}
    On the other hand,
    \begin{align*}
    \frac{d}{dt}\left[\frac{1}{2}M(t^2)-\int_{\Omega}\frac{F(x,t\mathscr{G}_{n,d})}{|x|^\beta}\right]\bigg|_{t=t_n}=0,
    \end{align*}
    from which we obtain
    \begin{align}\label{5.6}
    m(t^2_n)t^2_n&=\int_{\Omega}\frac{t_n\mathscr{G}_{n,d}\cdot\nabla F(x,t_n\mathscr{G}_{n,d})}{|x|^\beta}dx
    = \int_{\Omega}\frac{t_nG_{n,d}f_1(x,t_n\mathscr{G}_{n,d})}{|x|^\beta}dx.
    \end{align}
    By Remark \ref{rem5.2} and $(F_4)$, given $\delta>0$ there exists $s_\delta>0$ such that
    \begin{align}\label{5.7}
    u_1 f_1(x,u_1,0,\ldots,0)\geq (\eta_0 -\delta)e^{\alpha_0|u_1|^2},\ \ \ \forall\ x\in \Omega,\ \ |u_1|\geq s_\delta.
    \end{align}
    Lemma \ref{leest} shows that for any $0<\epsilon<1$, $t_n|\mathscr{G}_{n,d}|\geq s_\delta$ in $B_\frac{d}{\varpi(n)}(0)\subset\Omega$ for $n$ sufficiently large, where
$
    \varpi(n)=\exp\left\{\frac{n^{(1+\epsilon)/2}}{2}\right\}.
$
    Thus, by (\ref{5.6}) and (\ref{5.7}), we have
    \begin{align*}
    \begin{split}
    m(t^2_n)t^2_n\geq (\eta_0-\delta)\int_{B_\frac{d}{\varpi(n)}(0)} \frac{e^{\alpha_0|t_n G_{n,d}|^2}}{|x|^\beta}dx
    &=(\eta_0-\delta)\Big(\frac{d}{\varpi(n)}\Big)^{2-\beta}\int_{B_1(0)} \frac{e^{\beta_0|t_n \tilde{G}_n|^2}}{|x|^\beta}dx \\
    &=2\pi(\eta_0-\delta)\Big(\frac{d}{\varpi(n)}\Big)^{2-\beta}\int^1_0 e^{\alpha_0|t_n \tilde{G}_n(\sigma)|^2}\sigma^{1-\beta} d\sigma.
    \end{split}
    \end{align*}
    By performing the change of variable $\sigma=\tau^{\frac{2}{2-\beta}}$, we get
    \begin{align*}
    m(t^2_n)t^2_n\geq \frac{4\pi}{2-\beta}(\eta_0-\delta)\Big(\frac{d}{\varpi(n)}\Big)^{2-\beta}\int^{1}_0 e^{\frac{2\alpha_0|t_n \tilde{G}_n(\tau)|^2}{2-\beta}}\tau d\tau.
    \end{align*}
    Meanwhile, setting $\tau=e^{-t/2}$, we obtain
    \begin{align*}
    \begin{split}
    m(t^2_n)t^2_n\geq \frac{2\pi}{2-\beta}(\eta_0-\delta)\Big(\frac{d}{\varpi(n)}\Big)^{2-\beta}\int^{+\infty}_0 e^{\frac{\alpha_0|t_n y_n(t)|^2}{2\pi(2-\beta)}}e^{-t}dt.
    \end{split}
    \end{align*}
    Consequently,
    \begin{align}\label{5.8}
    \begin{split}
    m(t^2_n)t^2_n&\geq \frac{2\pi}{2-\beta}(\eta_0-\delta)\Big(\frac{d}{\varpi(n)}\Big)^{2-\beta}\int^{+\infty}_n e^{\frac{\alpha_0t_n^2(n-2\log n)}{2\pi(2-\beta)}}e^{-t}dt \\
    &=\frac{2\pi}{2-\beta}(\eta_0-\delta)d^{2-\beta}\exp\left\{\frac{\alpha_0t_n^2(n-2\log n)}{2\pi(2-\beta)}-(2-\beta)\log \varpi(n)-n\right\} \\
    &=\frac{2\pi}{2-\beta}(\eta_0-\delta)d^{2-\beta}\exp\left\{\frac{\alpha_0t_n^2(n-2\log n)}{2\pi(2-\beta)}-\frac{(2-\beta)n^{\frac{1+\epsilon}{2}}}{2}-n\right\} \\
    &=\frac{2\pi}{2-\beta}(\eta_0-\delta)d^{2-\beta}\exp\left\{\left[\frac{\alpha_0t_n^2}{2\pi(2-\beta)}-1\right]n-\frac{(2-\beta)n^{\frac{1+\epsilon}{2}}}{2}
    -\frac{\alpha_0t_n^2}{\pi(2-\beta)}\log n\right\}.
    \end{split}
    \end{align}
    From this
    \begin{align}\label{5.9}
    \begin{split}
    1\geq \frac{2\pi}{2-\beta}(\eta_0-\delta)d^{2-\beta}\exp\left\{t^2_n n\left[\frac{\alpha_0\big(1-\frac{2\log n}{ n}\big)}{2\pi(2-\beta)}-\frac{(2-\beta)n^{\frac{1+\epsilon}{2}}+2n}{2t^2_n n}-\frac{\log{\big (m(t^2_n)t^2_n\big)}}{t^2_n n} \right]\right\},
    \end{split}
    \end{align}
    thus, $\{t_n\}$ is bounded. Otherwise, noting that, from (\ref{1.4}), making use of the property of $M$ and $m$, we would have that
    \begin{align*}
    \begin{split}
    t^2_n n\left[\frac{\alpha_0\big(1-\frac{2\log n}{n}\big)}{2\pi(2-\beta)}-\frac{(2-\beta)n^{\frac{1+\epsilon}{2}}+2n}{2t^2_n n}-\frac{\log{\big (m(t^2_n)t^2_n\big)}}{t^2_n n} \right]\rightarrow+\infty,
    \end{split}
    \end{align*}
    which is a contradiction with (\ref{5.9}). Therefore $\{t_n\}$ has a subsequence convergent, from (\ref{5.5}), for some $t_0^2\geq \frac{2\pi(2-\beta)}{\alpha_0}$,\ $t_n\rightarrow t_0$. Moreover, using (\ref{5.8}), we must have $\frac{\alpha_0 t^2_0}{2\pi(2-\beta)}-1\leq 0$ and therefore,
    \begin{align}\label{5.10}
    t^2_n\rightarrow\frac{2\pi(2-\beta)}{\alpha_0}.
    \end{align}
    At this point, following arguments as in \cite{dmr1}, we are going to estimate (\ref{5.6}) more exactly. For this, in view of (\ref{5.7}), for $0<\delta<\eta_0$ and $n\in \mathbb{N}$ we set
    \begin{equation*}
    D_{n,\delta}:=\{x\in B_d(0):t_n G_{n,d}\geq s_\delta\}\ \ \mbox{and}\ \ E_{n,\delta}:=B_d(0)\backslash D_{n,\delta}.
    \end{equation*}
    Thus, by splitting the integral (5.6) on $D_{n,\delta}\ \mbox{and}\ E_{n,\delta}$, and using (\ref{5.7}), it follows that
    \begin{align}\label{5.11}
    \begin{split}
    M(t^2_n)t^2_n\geq &(\eta_0-\delta)\int_{B_d(0)}\frac{e^{\alpha_0 t^2_n G^2_{n,d}}}{|x|^\beta}dx-(\eta_0-\delta)\int_{E_{n,\delta}}\frac{e^{\alpha_0 t^2_n G^2_{n,d}}}{|x|^\beta}dx \\
    &+\int_{E_{n,\delta}}\frac{t_n G_{n,d}f_1(x,t_n \mathscr{G}_{n,d})}{|x|^\beta}dx.
    \end{split}
    \end{align}
    Since $G_{n,d}(x)\rightarrow 0$ for almost everywhere $x\in  B_d(0)$, we have that the characteristic functions $\chi_{E_{n,\delta}}$ satisfy
    \begin{equation*}
    \chi_{E_{n,\delta}}\rightarrow 1\ \ \mbox{a.e. in}\ \ B_d(0)\ \ \mbox{as}\ \ n\rightarrow+\infty.
    \end{equation*}
    Moreover, $t_n G_{n,d}<s_\delta$ in $E_{n,\delta}$. Thus, invoking the Lebesgue dominated convergence theorem, we obtain
    \begin{equation*}
    \int_{E_{n,\delta}}\frac{e^{\alpha_0 t^2_n G^2_{n,d}}}{|x|^\beta}dx\rightarrow\pi d^2 \ \ \mbox{and}\ \ \int_{E_{n,\delta}}\frac{t_n G_{n,d}f_1(x,t_n \mathscr{G}_{n,d})}{|x|^\beta}dx\rightarrow 0,\ \ \mbox{as}\ \ n\rightarrow+\infty.
    \end{equation*}
    Noting that
    \begin{equation*}
    \begin{split}
    \int_{B_d(0)}\frac{e^{\alpha_0 t^2_n G^2_{n,d}}}{|x|^\beta}dx=d^{2-\beta}\int_{B_1(0)}\frac{e^{\alpha_0 t^2_n \tilde{G}^2_n}}{|x|^\beta}dx=2\pi d^{2-\beta}\int^1_0{e^{\alpha_0 t^2_n \tilde{G}^2_n(\sigma)}}\sigma^{2-\beta}d\sigma,
    \end{split}
    \end{equation*}
    By performing the change of variable $\sigma=\tau^{\frac{2}{2-\beta}}$, we get
    \begin{align*}
    \int_{B_d(0)}\frac{e^{\alpha_0 t^2_n G^2_{n,d}}}{|x|^\beta}dx=\frac{4\pi}{2-\beta}d^{2-\beta}\int^{1}_0 e^{\frac{2\alpha_0|t_n \tilde{G}_n(\tau)|^2}{2-\beta}}\tau d\tau.
    \end{align*}
    Meanwhile, setting $\tau=e^{-t/2}$ and using (\ref{5.5}), we obtain
    \begin{align*}
    \begin{split}
    \int_{B_d(0)}\frac{e^{\alpha_0 t^2_n G^2_{n,d}}}{|x|^\beta}dx=&\frac{2\pi}{2-\beta}d^{2-\beta}\int^{+\infty}_0 e^{\frac{\alpha_0|t_n y_n(t)|^2}{2\pi(2-\beta)}}e^{-t}dt \\
    \geq& \frac{2\pi}{2-\beta}d^{2-\beta}\int^{+\infty}_0 e^{y^2_n(t)-t}dt.
    \end{split}
    \end{align*}
    Passing to limit in (\ref{5.11}), we obtain that
    \begin{align}\label{5.12}
    m\left(\frac{2\pi(2-\beta)}{\alpha_0}\right)\frac{2\pi(2-\beta)}{\alpha_0}\geq (\eta_0-\delta)\left[\frac{2\pi}{2-\beta} d^{2-\beta}(1+e)-\frac{2\pi}{2-\beta} d^{2-\beta}\right]=(\eta_0-\delta)\frac{2\pi e}{2-\beta} d^{2-\beta},
    \end{align}
    and doing $\delta\rightarrow 0^+$, we get $\eta_0\leq \frac { (2-\beta)^2m\big(\frac {2\pi(2-\beta)}{\alpha_0}\big)}{\alpha_0 d^{2-\beta} e }$, which contradicts $(F_4)$. Thus, this lemma is proved.
    \end{proof}

    Now, we establish an estimate for the minimax level.
    \begin{lemma}\label{lem5.3}
    If conditions $(M_1),\ (M_3)$ and  $(F_3)-(F_4)$ hold, then for small $\varepsilon$, it holds that
    \begin{align*}
    c_{M,\varepsilon}<\frac{1}{2}M\left(\frac{2\pi(2-\beta)}{\alpha_0}\right),
    \end{align*}
    where $c_{M,\varepsilon}$ is given as in (\ref{defmpl}).
    \end{lemma}

    \begin{proof}
    Since $\|\mathscr{G}_{n,d}\|=1$, as in the proof of Lemma \ref{lemgc2}, we have that $I(t\mathscr{G}_{n,d})\rightarrow -\infty$ as $t\rightarrow +\infty$. Consequently, $c_{M,\varepsilon}\leq\max_{t\geq 0}I_\varepsilon(t\mathscr{G}_{n,d}),\ \forall \ n\in \mathbb{N}$. Thus, from Lemma \ref{nl}, taking $\varepsilon$ sufficiently small, we can get what we desired.
    \end{proof}

    \begin{lemma}\label{lemms}
    If $f_i$ has critical growth at $\infty$,\  $(M_1)-(M_3)$ and $(F_1)-(F_4)$ satisfy, then for small $\varepsilon$, problem (\ref{P}) has one nontrivial mountain-pass type solution $U_{M,\varepsilon}$ at level $c_{M,\varepsilon}$, where $c_{M,\varepsilon}$ is given as in (\ref{defmpl}).
    \end{lemma}

    \begin{proof}
    From (\ref{5.1}) and Lemma \ref{lem4.1}, there exists a bounded Palais-Smale consequence $\{U_n\}$ for $I_\varepsilon$ at level $c_{M,\varepsilon}$. And, up to a subsequence, for some $U\in H^1_0(\Omega,\mathbb{R}^k)$, one has
    \begin{align}\label{5.14}
    U_n\rightharpoonup U\ \  \mbox{in}\ \ H^1_0(\Omega,\mathbb{R}^k); \qquad
    U_n \rightarrow U\ \   \mbox{in}\ \ L^s(\Omega,\mathbb{R}^k)  \ \mbox{for all}\ \ s\geq 1.
    \end{align}
    By Lemma \ref{lem4.1}, $\|U_n\|\leq C$ for some $C>0$, and $\|U\|\leq \lim\inf_{n\rightarrow\infty} \|U_n\|\leq C$. And by $(F_3)$, we have that for small $\varepsilon$, it holds
    \begin{align}\label{ms0}
    \begin{split}
    &\frac {1}{2\theta}\int_{\Omega}\bigg[\frac{U\cdot\nabla F(x,U)-2\theta F(x,U)}{|x|^\beta}-\varepsilon(2\theta-1) U\cdot H\bigg]dx \\
    &\geq\frac {1}{2\theta}\int_{\Omega}\frac{U\cdot\nabla F(x,U)-2\theta F(x,U)}{|x|^\beta}dx-\frac {\varepsilon (2\theta-1)\|U\| \|H\|_*}{2\theta}
    \geq 0.
    \end{split}
    \end{align}
    Next, we will make some claims as follows.

\noindent{\bfseries Claim 1.} \ \ \ $U\neq \mathbf{0}$.
    \begin{proof}
    Suppose by contradiction that $U\equiv \mathbf{0}$. Then Lemma \ref{lem4.4} show that $\int_{\Omega}U_n\cdot H dx\rightarrow 0$ and $\int_{\Omega}\frac{F(x,U_n)}{|x|^\beta}dx\rightarrow 0$, thus
    \begin{align*}
    \frac{1}{2}M(\|U_n\|^2)\rightarrow c_{M,\varepsilon}<\frac{1}{2}M\left(\frac{2\pi(2-\beta)}{\alpha_0}\right),
    \end{align*}
    then condition $(M_1)$ implies
    \begin{align*}
    \liminf_{n\rightarrow\infty}\|U_n\|^2<\frac{2\pi(2-\beta)}{\alpha_0}.
    \end{align*}
    By Lemma \ref{ms1}, we have $\|U_n\|^2\rightarrow 0$ and therefore $I_\varepsilon(U_n)\rightarrow 0$, what is absurd for $c_{M,\varepsilon}$ and we must have $U\neq \mathbf{0}$.
    \end{proof}

\noindent{\bfseries Claim 2.} \ \ \ Let $A:=\lim_{n\rightarrow \infty}\|U_n\|^2$, then $U$ is a weak solution of
    \begin{eqnarray*}
    \left\{ \arraycolsep=1.5pt
       \begin{array}{ll}
        -m(A)\Delta U=\frac{\nabla F(x,U)}{|x|^\beta}+\varepsilon H,\ \ &\mbox{in}\ \ \Omega,\\[2mm]
        U=0,\ \ &\mbox{on}\ \ \partial\Omega.
        \end{array}
    \right.
    \end{eqnarray*}

    \begin{proof}
    We define $C^\infty_0(\Omega,\mathbb{R}^k):=C^\infty_0(\Omega)\times \cdots\times C^\infty_0(\Omega)$. By $I_\varepsilon'(U_n)\rightarrow 0$ and Lemma \ref{lem4.4}, we see that
    \begin{align*}
    \begin{split}
    m(A)\int_{\Omega}\nabla U\cdot\nabla \Phi dx-\int_{\Omega}\frac{\Phi\cdot\nabla F(x,U)}{|x|^\beta}dx-\varepsilon\int_{\Omega}\Phi\cdot Hdx=0,\quad \forall\Phi\in C^\infty_0(\Omega,\mathbb{R}^k).
    \end{split}
    \end{align*}
    Since  $C^\infty_0(\Omega)$ is dense in $H^1_0(\Omega)$, then $C^\infty_0(\Omega,\mathbb{R}^k)$ is also dense in $H^1_0(\Omega,\mathbb{R}^k)$, and we conclude this claim.
    \end{proof}

\noindent{\bfseries Claim 3.} \ \ \ $A:=\lim_{n\rightarrow \infty}\|U_n\|^2<\|U\|^2+\frac{2\pi(2-\beta)}{\alpha_0}$.
    \begin{proof}
    Suppose by contradiction that $A\geq \|U\|^2+\frac{2\pi(2-\beta)}{\alpha_0}\geq \frac{2\pi(2-\beta)}{\alpha_0}$. Therefore, from (\ref{5.1}) and Lemma \ref{lem4.4}, we obtain
    \begin{align*}
    \begin{split}
    c_{M,\varepsilon}=&\lim_{n\rightarrow\infty}\left[I_\varepsilon(U_n)-\frac {1}{2\theta}\langle I'_\varepsilon(U_n),U_n\rangle_*\right] \\
    =&\frac {1}{2\theta}\lim_{n\rightarrow\infty}\Big[\theta M(\|U_n\|^2)-m(\|U_n\|^2)\|U_n\|^2\Big] \\
    &+\frac {1}{2\theta}\lim_{n\rightarrow\infty}\int_{\Omega}\frac{U_n\cdot\nabla F(x,U_n)-U\cdot\nabla F(x,U)}{|x|^\beta}dx \\
    &+\frac {1}{2\theta}\int_{\Omega}\bigg[\frac{U\cdot\nabla F(x,U)-2\theta F(x,U)}{|x|^\beta}-\varepsilon(2\theta-1)U\cdot H\bigg]dx, \\
    \end{split}
    \end{align*}
    thus, by $(M_3)$ and (\ref{ms0})
    \begin{align*}
    \begin{split}
    c_{M,\varepsilon}\geq&\frac {1}{2\theta}\lim_{n\rightarrow\infty}\Big[\theta M(\|U_n\|^2)-m(\|U_n\|^2)\|U_n\|^2\Big] \\
    &+\frac {1}{2\theta}\lim_{n\rightarrow\infty}\int_{\Omega}\frac{U_n\cdot\nabla F(x,U_n)-U\cdot\nabla F(x,U)}{|x|^\beta}dx \\
    =&\frac {1}{2\theta}\Big[\theta M(A)-m(A)A\Big]+\frac {1}{2\theta}\lim_{n\rightarrow\infty}\int_{\Omega}\frac{U_n\cdot\nabla F(x,U_n)-U\cdot\nabla F(x,U)}{|x|^\beta}dx \\
    \geq&\frac {1}{2\theta}\bigg[\theta M\left(\frac{2\pi(2-\beta)}{\alpha_0}\right)-m\left(\frac{2\pi(2-\beta)}{\alpha_0}\right)\frac{2\pi(2-\beta)}{\alpha_0}\bigg] \\
    &+\frac {1}{2\theta}\lim_{n\rightarrow\infty}\int_{\Omega}\frac{U_n\cdot\nabla F(x,U_n)-U\cdot\nabla F(x,U)}{|x|^\beta}dx \\
    =&\frac {1}{2}M\left(\frac{2\pi(2-\beta)}{\alpha_0}\right) \\
    &-\frac {1}{2\theta}\bigg[m\left(\frac{2\pi(2-\beta)}{\alpha_0}\right)\frac{2\pi(2-\beta)}{\alpha_0}-\lim_{n\rightarrow\infty}\int_{\Omega}\frac{U_n\cdot\nabla F(x,U_n)-U\cdot\nabla F(x,U)}{|x|^\beta}dx\bigg]. \\
    \end{split}
    \end{align*}
    Here, we assert
    \begin{align*}
    0\geq m\left(\frac{2\pi(2-\beta)}{\alpha_0}\right)\frac{2\pi(2-\beta)}{\alpha_0}-\lim_{n\rightarrow\infty}\int_{\Omega}\frac{U_n\cdot\nabla F(x,U_n)-U\cdot\nabla F(x,U)}{|x|^\beta}dx.
    \end{align*}
    Indeed, $\langle I'(U_n),U_n\rangle_*\rightarrow 0$, (\ref{dch}) and Claim 2 indicate that
    \begin{align*}
    \begin{split}
    0=&m(A)A-\lim_{n\rightarrow\infty}\int_{\Omega}\frac{U_n\cdot\nabla F(x,U_n)}{|x|^\beta}dx-\varepsilon\int_{\Omega}U\cdot H dx, \\
    0=&m(A)\|U\|^2-\int_{\Omega}\frac{U\cdot\nabla F(x,U)}{|x|^\beta}dx-\varepsilon\int_{\Omega}U\cdot H dx.
    \end{split}
    \end{align*}
    Subtracting the second equality from the first one, by $(M_2)$ which implies $m(t)$ and $m(t)t$ are nondecreasing for $t\geq 0$, we get
    \begin{align*}
    \begin{split}
    0&=m(A)(A-\|U\|^2)-\lim_{n\rightarrow\infty}\int_{\Omega}\frac{U_n\cdot\nabla F(x,U_n)-U\cdot\nabla F(x,U)}{|x|^\beta}dx \\
    &\geq m(A-\|U\|^2)(A-\|U\|^2)-\lim_{n\rightarrow\infty}\int_{\Omega}\frac{U_n\cdot\nabla F(x,U_n)-U\cdot\nabla F(x,U)}{|x|^\beta}dx \\
    &\geq m\left(\frac{2\pi(2-\beta)}{\alpha_0}\right)\frac{2\pi(2-\beta)}{\alpha_0}-\lim_{n\rightarrow\infty}\int_{\Omega}\frac{U_n\cdot\nabla F(x,U_n)-U\cdot\nabla F(x,U)}{|x|^\beta}dx.
    \end{split}
    \end{align*}
    This shows the assertion. Noting Lemma \ref{lem5.3}, we conclude
    \begin{align*}
    \frac {1}{2}M\left(\frac{2\pi(2-\beta)}{\alpha_0}\right)\leq c_{M,\varepsilon}<\frac {1}{2}M\left(\frac{2\pi(2-\beta)}{\alpha_0}\right),
    \end{align*}
    which is absurd. Thus this claim is proved.
    \end{proof}

\noindent{\bfseries Claim 4.} \ \ \ $A:=\lim_{n\rightarrow \infty}\|U_n\|^2=\|U\|^2$.
    \begin{proof}
    By using semicontinuity of norm, we have $\|U\|^2\leq A$. We are going to show that the case $\|U\|^2< A$ can not occur. Indeed, if $\|U\|^2< A$, defining $Z_n=\frac{U_n}{\|U_n\|}$ and $Z_0=\frac{U}{A^{1/2}}$, we have $Z_n\rightharpoonup Z_0\ \mbox{in} \ H^1_0(\Omega,\mathbb{R}^k)$ and $\|Z_0\|<1$. Thus, by Lemma \ref{lemtm4}
    \begin{align}\label{5.16}
    \sup_n \int_{\Omega}\frac{e^{p|Z_n|^2}}{|x|^\beta}dx<\infty,\ \ \forall p<\frac{2\pi(2-\beta)} {1-\|Z_0\|^2}.
    \end{align}
    Since $A=\frac{A-\|U\|^2}{1-\|Z_0\|^2}$, from Claim 3 which follows that $A<\frac{2\pi(2-\beta)}{\alpha_0(1-\|Z_0\|^2)}$.
    Thus, there exists $\zeta>0$ such that $\alpha_0\|U_n\|^2<\zeta<\frac{2\pi(2-\beta)}{1-\|Z_0\|^2}$ for $n$ sufficiently large. For $q>1$ close to 1 and $\alpha>\alpha_0$ close to $\alpha_0$ we still have $q\alpha\|U_n\|^2< \zeta<\frac{2\pi(2-q\beta)}{1-\|Z_0\|^2}$ with $q\beta<2$, and provoking (\ref{5.16}), for some $C>0$ and $n$ large enough, we conclude that
    \begin{align*}
    \int_{\Omega}\frac{e^{q\alpha|U_n|^2}}{|x|^{q\beta}}dx\leq \int_{\Omega}\frac{e^{\zeta|Z_n|^2}}{|x|^{q\beta}}dx\leq C.
    \end{align*}
    Hence, using (\ref{3.1}), (\ref{5.14}) and the H\"{o}lder's inequality, we get
    \begin{align*}
    \begin{split}
    \left|\int_{\Omega} \frac{(U_n-U)\cdot\nabla F(x,U_n)}{|x|^\beta} dx\right|&\leq C_1\int_{\Omega}|U_n-U|\frac{e^{\alpha|U_n|^2}}{|x|^\beta} dx \\
    &\leq C_1\|U_n-U\|_{\frac{q}{q-1}}\left(\int_{\Omega}\frac{e^{q\alpha|U_n|^2}}{|x|^{q\beta}} dx\right)^{1/q} \\
    &\leq C_2\|U_n-U\|_{\frac{q}{q-1}}\rightarrow 0,
    \end{split}
    \end{align*}
    as $n\rightarrow\infty.$ Since $\langle I'_\varepsilon(U_n),U_n-U\rangle_*\rightarrow 0$, by (\ref{dch}), it follows that $m(\|U_n\|^2)\langle U_n,U_n-U\rangle\rightarrow 0$. On the other hand,
    \begin{align*}
    \begin{split}
    m(\|U_n\|^2)\langle U_n,U_n-U\rangle&=m(\|U_n\|^2)\|U_n\|^2-m(\|U_n\|^2)\int_{\Omega}\nabla U_n\cdot\nabla Udx \\
    &\rightarrow m(A)A-m(A)\|U\|^2,
    \end{split}
    \end{align*}
    which implies that $A=\|U\|^2$ what is absurd. Thus, this claim is proved.
    \end{proof}

\noindent{\bfseries Finalizing the proof of Lemma \ref{lemms}}: Since $H^1_0(\Omega,\mathbb{R}^k)$ is uniformly convex Banach space and (\ref{5.14}), Claim 4, by Radon's Theorem, $U_n\rightarrow U$ in $H^1_0(\Omega,\mathbb{R}^k)$. Hence, by (\ref{5.1}): $I_\varepsilon'(U_n)\rightarrow 0$ and Lemma \ref{lem4.4}, we have
    \begin{equation*}
    m(\|U\|^2)\int_{\Omega}\nabla U \cdot\nabla \Phi=\int_{\Omega}\frac{\Phi\cdot\nabla {F(x,U)}}{|x|^\beta} dx+\varepsilon\int_{\Omega}H \cdot \Phi dx,\ \ \forall \Phi\in C^\infty_0(\Omega,\mathbb{R}^k).
    \end{equation*}
    Since $C^\infty_0(\Omega,\mathbb{R}^k)$ is dense in $H^1_0(\Omega,\mathbb{R}^k)$, we conclude that $U_{M,\varepsilon}:=U$ is a Mountain-pass type solution for problem (\ref{P}) with $I_\varepsilon(U_{M,\varepsilon})=c_{M,\varepsilon}>0$ and according to Claim 1, the proof is complete.
    \end{proof}

    Finally, let us to find out a minimum type solution $V_{0,\varepsilon}$ with $I_\varepsilon(V_{0,\varepsilon})=c_{0,\varepsilon}<0$, where $c_{0,\varepsilon}$ is given as in (\ref{defc0}).

    \begin{lemma}\label{lemms2}
    For small $\varepsilon$, problem (\ref{P}) has a nontrivial minimum type solution $V_{0,\varepsilon}$ with $I_\varepsilon(V_{0,\varepsilon})=c_{0,\varepsilon}<0$.
    \end{lemma}

    \begin{proof}
    Let $\rho_\varepsilon$ be as in Lemma \ref{lemgc1}. Then we can choose $\varepsilon$ sufficiently small such that
    \begin{align*}
    \rho_\varepsilon<\left(\frac{2\pi(2-\beta)}{\alpha_0}\right)^{1/2}.
    \end{align*}
    Since $\overline{B}_{\rho_\varepsilon}$ is a complete metric space with the metric given by the norm of $H^1_0(\Omega,\mathbb{R}^k)$, convex and the functional $I_\varepsilon$ is of class $C^1$ and bounded below on $\overline{B}_{\rho_\varepsilon}$, by the Ekeland variational principle, there exists a sequence $\{V_n\}$ in $\overline{B}_{\rho_\varepsilon}$ such that
    \begin{equation*}
    I_\varepsilon(V_n)\rightarrow c_{0,\varepsilon}=\inf_{\|V\|\leq \rho_\varepsilon}I_\varepsilon(V)<0\ \ \mbox{and}\ \ I'_\varepsilon(V_n)\rightarrow 0.
    \end{equation*}
    Observing that
    \begin{align*}
    \|V_n\|^2\leq \rho_\varepsilon^2<\frac{2\pi(2-\beta)}{\alpha_0},
    \end{align*}
    by Lemma \ref{ms1}, there exists a strongly convergent subsequence and therefore, for some $V_{0,\varepsilon}$, $V_n\rightarrow V_{0,\varepsilon}$ strongly in $H^1_0(\Omega,\mathbb{R}^k)$. Consequently, we have $V_{0,\varepsilon}$ is a minimum type solution of problem (\ref{P}) with $I_\varepsilon(V_{0,\varepsilon})=c_{0,\varepsilon}<0$. We claim $V_{0,\varepsilon}\neq \mathbf{0}$. Indeed, suppose by a contradiction $V_{0,\varepsilon}= \mathbf{0}$, then $0>c_{0,\varepsilon}=I_\varepsilon(V_{0,\varepsilon})=I_\varepsilon(\mathbf{0})=0$, what is absurd and this lemma is proved.
    \end{proof}

\noindent{\bfseries Proof of Theorem \ref{thm1.3}.}
    By Lemmas \ref{lemms} and \ref{lemms2}, there exists $\varepsilon_c>0$ such that for each $0<\varepsilon<\varepsilon_c$, there exist nontrivial critical points $U_{M,\varepsilon}$ for $I_\varepsilon$ at level $c_{M,\varepsilon}$ and $V_{0,\varepsilon}$ for $I_\varepsilon$ at level $c_{0,\varepsilon}$. In the end, we claim $U_{M,\varepsilon}\neq V_{0,\varepsilon}$. Suppose by contradiction that $ U_{M,\varepsilon}\equiv V_{0,\varepsilon}$, then $0>c_{0,\varepsilon}=\lim_{n\rightarrow\infty}I_\varepsilon(V_n)=I(V_{0,\varepsilon})
    =I(U_{M,\varepsilon})=\lim_{n\rightarrow\infty}I_\varepsilon(U_n)=c_{M,\varepsilon}>0$, what is absurd. Thus, the proof of Theorem \ref{thm1.3} is complete.
    \qed

\medskip

\noindent{\bfseries Acknowledgements.}
The authors have been supported by NSFC 11971392, Natural Science Foundation of Chongqing, China cstc2019jcyjjqX0022  and Fundamental Research
Funds for the Central Universities XDJK2019TY001.


\begin{thebibliography}{99}


\bibitem{ad}
Adimurthi, {\em  Existence of positive solutions of the semilinear Dirichlet problem with critical growth for the $N$-Laplacian}, Ann. Sc. Norm. Super. Pisa Cl.Sci. {\bf 17} (1990), 393-413.

\bibitem{as}
Adimurthi, K. Sandeep, {\em A singular Moser-Trudinger embedding andits applications}, NoDEA Nonlinear Differential Equations Appl. {\bf 13} (2007), 585-603.

\bibitem{asy}
Adimurthi, P. N. Skikanth, S.L. Yadava, {\em Phenomena of critical exponent in ${\mathbb {R}}^2$}, Proc. Roy. Soc. Edinburgh Sect. A, {\bf 119} (1991), 19-25.

\bibitem{ay}
Adimurthi, S. L. Yadava, {\em Multiplicity results for semilinear elliptic equations in bounded domain of ${\mathbb {R}}^2$ involving critical exponent}, Ann. Sc. Norm. Super. Pisa, {\bf 17} (1990), 481-504.

\bibitem{adiyang}
Adimurthi, Y. Yang, {\em An interpolation of Hardy inequality and Trudinger-Moser inequality in $\mathbb{R}^N$ and its applications}, Int. Math. Res. Notices. {\bf 13} (2010), 2394-2426.

\bibitem{am}
F. Albuquerque, E. Medeiros, {\em An elliptic equation involving exponential critical growth in $\mathbb{R}^2$}, Adv. Nonlinear Stud. {\bf 15} (2015), no. 1, 23-37.

\bibitem{acm}
C. O. Alves, F. J. S. A. Corr\^{e}a, T.F. Ma, {\em Positive solutions for a quasilinear elliptic equation of Kirchhoff type}, Comput. Math. Appl. {\bf 49} (2005), 85-93.


\bibitem{bl}
H. Berestycki, P. L. Lions, {\em Nonlinear scalar field equations, I. Existence of ground state}, Arch. Ration Mech. Anal. {\bf 82} (1983), 313-346.

\bibitem{bl2}
H. Br\'{e}zis, E. Lieb, {\em A relation between pointwise convergence of functions and convergence of functionals}, Proc. Amer. Math. Soc. {\bf 88} (1983), 486-490.


\bibitem{chen}
W. Chen, {\em Existence of solutions for fractional p-Kirchhoff type equations with a generalized Choquard nonlinearities}, Topol. Methods Nonlinear Anal.  {\bf 54} (2019), 773-791.

\bibitem{cy2}
W. Chen, F. Yu, {\em On a nonhomogeneous Kirchhoff-type elliptic problem with critical exponential in dimension two}, Appl. Anal., (2020), 1-16.

\bibitem{c2}
B. Cheng, {\em New existence and multiplicity of nontrivial solutions for nonlocal elliptic Kirchhoff type problems}, J. Math. Anal. Appl. {\bf 394} (2012), 488-495.

\bibitem{ddr}
D. G de Figueiredo, J.M. do \'{O}, B. Ruf, {\em On an inequality by Trudinger and J. Moser and related elliptic equations}, Comm. Pure Appl. Math. {\bf 55} (2002), 135-152.

\bibitem{ddr2}
D. G de Figueiredo, J.M. do \'{O}, B. Ruf, {\em Semilinear elliptic systems with exponential nonlinearities in two dimensions},  Adv. Nonlinear Stud. {\bf 6} (2006), no. 2, 199-213.

\bibitem{dmr1}
D. G. de Figueiredo, O.H. Miyagaki, B. Ruf, {\em Elliptic equations in ${\mathbb {R}}^2$ with nonlinearities in the critical growth range},  Calc. Var. Partial Differential Equations {\bf 3} (1995), 139-153.



\bibitem{ds}
M. de Souza, {\em On a singular class of elliptic systems involving critical growth in $\mathbb{R}^2$}, Nonlinear Analysis: Real World Applications, {\bf 12} (2011), 1072-1088.


\bibitem{doms}
J. M. do \'{O}, E. Medeiros, U. Severo, {\em On a quasilinear nonhomogeneous elliptic equation with critical growth in $\mathbb{R}^N$}, J. Differential Equations, {\bf 246} (2009), 1363-1386.


\bibitem{do}
J. M. do \'{O}, {\em Semilinear Dirichlet problems for the $N$-Laplacian in ${\mathbb {R}}^N$ with nonlinearities in the critical growth range}, Differential Integral Equations, {\bf 9} (1996), 967-979.

\bibitem{dms}
J. M. do \'{O}, E. Medeiros, U.B. Severo, {\em A nonhomogeneous elliptic problem involving critical growth in dimension two}, J. Math. Anal. Appl. {\bf 345} (2008), 286-304.


\bibitem{fs}
G. M. Figueiredo, U.B. Severo, {\em Ground state solution for a Kirchhoff problem with exponential critical growth}, Milan J. Math. Vol. {\bf 84} (2016), 23-39.

\bibitem{fiscellaValdinoci}
A. Fiscella, E. Valdinoci, {\em A critical Kirchhoff type problem involving a nonlocal operator}, Nonlinear Anal. {\bf 94} (2014), 156-170.

\bibitem{hezou}
X. He, W. Zou, {\em Infinitely many positive solutions of Kirchhoff type problems}, Nonlinear Anal. {\bf 70(3)} (2009), 1407-1414.

\bibitem{hezou2}
X. He, W. Zou, {\em Multiplicity solutions of for a class of Kirchhoff type problems}, Acta Mathematicae Applicatae Sinica, English Series {\bf 26(3)} (2010), 387-394.


\bibitem{lamlu4}
N. Lam, G. Lu, {\em Existence and multiplicity of solutions to equations of $N$-Laplacian type with critical exponential growth in $\mathbb{R}^N$}, J. Funct. Anal. {\bf 262} (2012), no. 3, 1132-1165.


%\bibitem{lamlu1}
%N. Lam, G. Lu, {\em Existence of nontrivial solutions to polyharmonic equations with subcritical and critical exponential growth}, Discrete Contin. Dyn. Syst. {\bf 32} (2012), no. 6, 2187-2205.


%\bibitem{lamlu2}
%N. Lam, G. Lu, {\em $N-$Laplacian equations in $\mathbb{R}^N$ with subcritical and critical growth without the Ambrosetti-Rabinowitz condition}, Adv. Nonlinear Stud. {\bf 13} (2013), no. 2, 289-308.


%\bibitem{lamlu3}
%N. Lam, G. Lu, {\em Elliptic equations and systems with subcritical and critical exponential growth without the Ambrosetti-Rabinowitz condition}, J. Geom. Anal. {\bf 24} (2014), no. 1, 118-143.



\bibitem{k}
G. Kirchhoff, Mechanik, Teubner, Leipzig, 1883.

\bibitem{m}
J. Moser, {\em A sharp form of an inequality by N. Trudinger}, Ind. Univ. Math. J. {\bf 20} (1971), 1077-1092.

\bibitem{nt}
D. Naimen, C. Tarsi, {\em Multiple solutions of a Kirchhoff type elliptic problem with the Trudinger-Moser growth},  Advances in Differential Equations. {\bf 22} (2017), 983-1012.


\bibitem{ra}
P. Rabelo, {\em Elliptic Systems involving critical growth in dimension two}, Commun. Pure Appl. Anal. {\bf 8} (2009), 2013-2035.


\bibitem{Trudinger}
N. Trudinger, {\em On imbeddings into Orlicz spaces and some applications}, J. Math. Mech. {\bf 17} (1967), 473-483.


\bibitem{y2012}
Y. Yang, {\em Existence of positive solutions to quasi-linear elliptic equations with exponential growth in the whole Euclidean space}, J. Funct. Anal. {\bf 262} (2012), 1679-1704.

\end{thebibliography}
\end{document}